\documentclass{article}
\usepackage[utf8]{inputenc}
\usepackage[affil-it]{authblk}
\usepackage{amsfonts,amsmath,amssymb} %amssymb for lesssim
\usepackage{amsthm}
\usepackage[bibstyle=authortitle,backend=biber, style=numeric]{biblatex}
\usepackage{hyperref}
\usepackage{geometry}
\usepackage{comment}
\usepackage{color}
\usepackage{esint}
\usepackage{bbm} %indicator function \[ \mathbbm{1} \]
\usepackage{graphicx}
\usepackage{xcolor}
\usepackage{tikz}
\usetikzlibrary{mindmap}
\graphicspath{ {./Brenier/} }
\theoremstyle{definition} \newtheorem{definition}{Definition}[section]
\theoremstyle{definition} \newtheorem{remark}[definition]{Remark}
\theoremstyle{plain} \newtheorem{lemma}[definition]{Lemma}
\theoremstyle{plain} \newtheorem{proposition}[definition]{Proposition}
\theoremstyle{plain} \newtheorem{theorem}[definition]{Theorem}
\theoremstyle{plain} \newtheorem{corollary}[definition]{Corollary}
\theoremstyle{definition} 
\theoremstyle{definition}

\DeclareMathOperator{\BV}{BV}
\DeclareMathOperator{\dive}{div}

\newcommand{\R}{\mathbb{R}}
\newcommand{\Q}{\mathbb{Q}}
\newcommand{\N}{\mathbb{N}}

\newcommand{\M}{\mathcal{M}}

\renewcommand{\L}{\mathscr L}

\renewcommand{\L}{\mathcal L}
\renewcommand{\H}{\mathcal H}

\numberwithin{equation}{section}
\def\Xint#1{\mathchoice
	{\XXint\displaystyle\textstyle{#1}}%
	{\XXint\textstyle\scriptstyle{#1}}%
	{\XXint\scriptstyle\scriptscriptstyle{#1}}%
	{\XXint\scriptscriptstyle\scriptscriptstyle{#1}}%
	\!\int}
\def\XXint#1#2#3{{\setbox0=\hbox{$#1{#2#3}{\int}$ }
		\vcenter{\hbox{$#2#3$ }}\kern-.6\wd0}}

\def\dashint{\Xint-}

\title{Geometric interpretation of the vanishing Lie Bracket for two-dimensional rough vector fields}
\author{\thanks{Max-Planck Institute for Mathematics in the Sciences} Rebucci A.\footnote{{annalaura.rebucci@mis.mpg.de}}, Zizza M., \footnote{martina.zizza@mis.mpg.de}}

\date{\today}

\begin{document}

\maketitle
\begin{abstract}
    In this paper, we prove that if $X,Y$ are continuous, Sobolev vector fields with bounded divergence on the real plane and $[X,Y]=0$, then their flows commute. In particular, we improve the previous result of \cite{ColomboTione}, where the authors require the additional assumption of the weak Lie differentiability on one of the two flows. We also discuss possible extensions to the $\BV$ setting.
\end{abstract}

\medskip

\noindent Keywords: Frobenius's Theorem, Lie Bracket, rough vector fields, Regular Lagrangian Flows. \\

\noindent MSC2020: 35A24, 35D30, 35F35, 35Q49.
 \\
\section{Introduction}
%{\color{violet}Cose importanti: 1) Finire di scrivere la dimostrazione del teorema 4.1; 2) Decidere se aggiungere la definizione di Lie bracket lungo il flusso e l'altra implicazione o se lasciare tutto cosi}
The aim of this paper is to extend a classical result on the commutativity of flows of smooth vector fields (in the Euclidean space) to the weak setting of Regular Lagrangian Flows. Recall that, if $X,Y\in C^\infty(\R^n,\R^n)\cap L^\infty(\R^n,\R^n)$ are two smooth vector fields, then their flows $\phi^X_t,\phi^Y_s$ are the solutions of the following Cauchy Problem
\begin{equation*}
    \begin{cases}
        \partial_t\phi^X_t(x)=X(\phi^X_t(x)),\\
        \phi^X_0(x)=x,
    \end{cases}
\end{equation*}
$\forall x\in\R^n$, (and the same for the flow $\phi^Y_s$). A classical result of differential geometry, essential to Frobenius' theorem, states that
\begin{equation}\label{eq:commutativity-classical}
  \phi_t^X \circ \phi_s^Y=\phi_s^Y \circ \phi_t^X , \quad \forall t,s \in \R \qquad \iff \qquad [X,Y]=0,
\end{equation}
where $[X,Y]$ can be defined pointwise as $[X,Y]=DYX-DXY$. The equivalence \eqref{eq:commutativity-classical} was proved in \cite{RampazzoSussmann} for locally Lipschitz vector fields, and was recently extended in \cite{RigoniStepanov} to the case where one vector field is Lipschitz continuous and the other is Sobolev regular (but admitting a Regular Langrangian Flow). The validity of \eqref{eq:commutativity-classical} to the weak setting of Regular Lagrangian Flows was first studied in \cite{ColomboTione}, wherein the authors prove that, if $X,Y\in W^{1,p}_{\text{loc}}(\R^n,\R^n)\cap L^\infty(\R^n,\R^n)$ with bounded divergence, then
 \begin{equation}\label{eq:defcommu}
 \left( \phi^X_t\circ\phi^Y_s \right)(z)=\left( \phi^Y_s\circ\phi^X_t \right)(z),  \quad\forall s,t\in\R,\text{ for a.e. } z\in\R^n  
 \end{equation}
 \begin{equation*}
     \iff 
 \end{equation*}
\begin{align*}
    [X,Y]=0 \text{ and }\phi^X_t \text{ is weakly (Lie) differentiable w.r.t. $Y$ with locally bounded derivative {for all $t \in \R$}.}
\end{align*}

For rough vector fields as in the above, the flows of the two vector fields are given by Regular Lagrangian Flows as in \cite{DiPerna:Lions},\cite{Ambrosio:BV} (see also Section \ref{S:preliminaries}). Moreover, we remark that the condition $\forall s,t\in\R$, for a.e. $z\in\R^n$ in \eqref{eq:defcommu} means that for every $s,t$ fixed, there exists $\Omega_{s,t}$ such that $\L^{\color{black}n}(\Omega_{s,t}^c)=0$ and the flows commute in $\Omega_{s,t}$. Hence, it is not restrictive to consider a bounded time interval, i.e. $t\in(-T,T)$ for some $T>0$.  We finally recall that the notion of weak Lie differentiability introduced in \cite{ColomboTione} has to be intended in the following sense: %{\color{red}Ricordiamoci di uniformare la notazione del prodotto scalare. Io sono per levare quelle parentesi rotonde orrende}
\begin{definition}\label{def:Lie:derivative}
    Let $F\in L^1(\R^n,\R^n)$ and let $Y\in L^\infty(\R^n,\R^n)$ be a vector field with $\dive Y\in  L^\infty(\R^n,\R)$. We say that $F$ is weakly (Lie) differentiable w.r.t. $Y$ if there exists $f\in L^1(\R^n,\R^n)$ such that for every $\varphi\in C^1_c(\R^n,\R^n)$ the following holds
    \begin{equation}\label{eq:Lie:der}
    \int_{R^n}\left(F(z),D\varphi(z) Y(z)\right)+\dive(Y)(z)\left(F(z),\varphi(z)\right)dz=-\int_{\R^n}\left(f(z),\varphi(z)\right)dz.
    \end{equation}
    We say that $f$ is the \emph{weak (Lie) derivative} of $F$ in the direction of $Y$. Moreover, if $f\in L^\infty_{\text{loc}}(\R^n,\R^n)$ we say that $F$ is weakly (Lie) differentiable w.r.t. $Y$ with locally bounded derivative. 
\end{definition}
In particular, if $F\in W^{1,p}_{\text{loc}}(\R^n,\R^n)$, then $F$ is weakly (Lie) differentiable. Indeed, the weak Lie derivative is the directional derivative with the direction given by the vector field $Y$. 

We point out that the assumption of weak Lie differentiability is a non-trivial requirement in the non-smooth setting. Indeed, in the case of rough regularity of the vector field (being either Sobolev or $\BV$), the flow may not be differentiable, namely it may not belong to any Sobolev/$\BV$ space, as first proved in \cite{jab16} through a random construction. We also mention \cite{Eliodiff}, where the author provides an example where the RLF fails to be Sobolev/$\BV$. Finally, it is worth mentioning that the validity of the equivalence \eqref{eq:commutativity-classical} in this weaker setting is a non-trivial problem. Indeed, in \cite{RigoniStepanov}, the authors construct a counterexample showing that \eqref{eq:commutativity-classical} cannot be extended to general a.e. differentiable vector fields admitting a.e. unique flows.

\bigskip

In this paper, we propose a purely geometric approach that allows us to remove the additional assumption of the weak Lie differentiability in the case of two-dimensional vector fields with bounded divergence. More precisely, we prove that the classical equivalence \eqref{eq:commutativity-classical} holds in dimension $n=2$, with the further assumption of the continuity of the two vector fields. Moreover, in Section \ref{Sect:BV} we explain how to generalize one of the implications in \eqref{eq:commutativity-classical} to the $\BV$ setting.

\smallskip

We start by describing the heuristics of our approach, in order to show in particular how  to drop the weak Lie differentiability assumption on the flow of $X$. Our construction draws inspiration from \cite{Eliodiff}. The core idea is that, if the vector field $X\in W^{1,p}(\R^2,\R^2)\cap C^0(\R^2,\R^2)$ is bounded, divergence-free and uniformly bounded from below, then its flow $\phi^X_t$ belongs to the Sobolev space $ W^{1,p}$. The key theorem in \cite{Eliodiff} is the following:
\begin{theorem}\label{thm:Marconi}[Theorem 1.4 in \cite{Eliodiff}]
    Let $X\in W^{1,p}(\R^2,\R^2)$ be a continuous divergence-free vector field with bounded support and let $\Omega$ be an open ball of radius $R>0$ such that there exists $\delta>0$ and $\mathbf e\in\mathbb{S}^1$ for which $X\cdot \mathbf e>\delta$ in $\Omega$. Then for every $t>0$ the Regular Lagrangian Flow of $X$ has a representative $\phi^X_t\in C^0(\Omega)\cap W^{1,p}(\Omega).$
\end{theorem}
The proof of the previous result is based on the underlying Hamiltonian structure of the vector field, and therefore the set $\Omega$, which in \cite{Eliodiff} is a ball, must be simply connected to guarantee the existence of a Lipschitz Hamiltonian $H$ such that $X=\nabla^\perp H$. We remark also that, if we consider $\mathbf e$ as in the statement of Theorem \ref{thm:Marconi}, then we have $|X|\geq \delta >0$. A version of Theorem \ref{thm:Marconi} in the $\BV$ setting is also available in \cite{Eliodiff}.

\smallskip

By Theorem \ref{thm:Marconi}, it is clear that if $|X|\geq c>0$ in some $\Omega\in\R^2$, and the set $\Omega$ is invariant for the flows $\phi^X_t,\phi^Y_s$ (meaning that the flows always remain in the set $\Omega$), then $\phi^X_t$ is weakly Lie differentiable in $\Omega$, since it is Sobolev, and therefore the two flows commute in $\Omega$. On the other hand, if $|X|=0$, then the corresponding (Regular Lagrangian) flow is the identity map (at least for a.e. $x\in\R^2$) and the commutativity simply follows by showing that $X\circ \phi^Y_s=0$ for every $s$, i.e. that the vector field $X$ vanishes along the trajectories of $Y$. 

The proof of our main result relies on finding regions of the plane invariant under the flows $\phi^X_t,\phi^Y_s$ where the two different behaviours occur. In this context, the condition $[X,Y]=0$ prescribes very strict geometric properties on the flows of the two vector fields (see in particular Equation \eqref{eq:Lie:bracket:intro} below). 

\bigskip

 Finally, we remark that the study of the Lie Bracket for rough vector fields is of interest in the context of Magneto Hydrodynamics, see, for example, \cite{BrenierMHD}. Furthermore, the investigation of the commutativity of the flows in the non-smooth setting, together with the notion of Lie bracket for rough vector fields, arises naturally in various research fields, see, for instance, \cite{RigoniStepanov},\cite{DelNinBonicatto}. 

\subsubsection*{The Hamiltonian case}
In the first part of our work, Section \ref{Sect:hamiltonians}, we restrict ourselves to the case of Hamiltonian dynamics, i.e. we assume $\dive X=0$, $\dive Y=0$, $n=2$. Our main contribution here is the following result.

\begin{proposition}\label{thm:main:hamiltonian}
    Let $X,Y\in W^{1,p}_{loc}(\R^2,\R^2)\cap C^0(\R^2,\R^2)$ be two bounded vector fields and assume that $\dive X=\dive Y=0$ almost everywhere. Then 
    \begin{equation}\label{eq:equivalence:BV}
        \qquad [X,Y]=0 \iff \phi_t^X \circ \phi_s^Y=\phi_s^Y \circ \phi_t^X , \quad \forall t,s \in \R,\quad\text{for a.e. in }\R^2 .
    \end{equation}
\end{proposition}
The main idea behind the proof of Proposition \ref{thm:main:hamiltonian} is to exploit the well-known fact that, if $X \in L^\infty(\R^2,\R^2)$ with $\dive X=0$, then there exists a Lipschitz Hamiltonian $H:\R^2 \rightarrow \R$ such that
\begin{equation*}
    X=\nabla^\perp H=(-\partial_2 H,\partial_1 H).
\end{equation*}
Formally, the flow preserves the Hamiltonian, i.e. the trajectories of the flows are contained in the level sets of $H$. This property was exploited in \cite{ABC} to split the uniqueness problem for the continuity equation into a family of one-dimensional
problems on the level sets of the Hamiltonian $H$. A technique introduced in \cite{ABC} allowed the authors to characterize the Hamiltonians (and corresponding vector fields) for which one has uniqueness to the continuity equation in the class of bounded solutions. 

It is worth mentioning that the approach described above was also used in \cite{BianchiniGusev} to show the uniqueness of the Regular Lagrangian Flow for steady
nearly incompressible vector fields with bounded variation. The result obtained in \cite{BianchiniGusev} was then later extended in \cite{BianchiniBonicattoGusev} to nearly incompressible fields with bounded variation. Once again, the proof in \cite{BianchiniBonicattoGusev} relies on splitting the transport equation related to the vector field over a suitable partition of the plane in the spirit of \cite{ABC}.

The key observation in our analysis is the following: Given two smooth Hamiltonian vector fields $X=\nabla^\perp H,  Y=\nabla^\perp K$, we have
\begin{equation*}
    [X,Y]=0 \iff \lbrace H,K\rbrace=\text{const},
\end{equation*}
where $\lbrace \cdot,\cdot\rbrace$ denotes the Poisson Bracket.  In the two-dimensional case, the vanishing Lie bracket corresponds in particular to the following property:
\begin{equation}\label{eq:property:scalar:product}
   X\cdot Y^\perp=\text{const}.
\end{equation}
%as shown in Lemma \ref{lem:eq:for:bracket} below. 
Indeed, the following equation holds \cite{Arnold:khesin}:

    \begin{equation}\label{eq:Lie:bracket:intro}
        [X,Y]= \dive YX-\dive X Y -\nabla^\perp (X\cdot Y^\perp).
        \end{equation}
In Section \ref{Sect:hamiltonians}, we show that the same equivalence between the vanishing Lie bracket and property \eqref{eq:property:scalar:product} still holds if we interpret the Lie bracket in \eqref{eq:Lie:bracket:intro} in the distributional sense. Hence, we are left to prove that $X\cdot Y^\perp=\text{const} \Rightarrow \phi_t^X \circ \phi_s^Y=\phi_s^Y \circ \phi_t^X$, for every $ t,s \in \R$, for a.e. in $x \in \R^2$ and therefore we can study separately the two different cases $ X\cdot Y^\perp\not=0$, and $ X\cdot Y^\perp=0$. In the first case, the commutativity of the flows is an almost direct consequence of the results proved in \cite{Eliodiff} (see also \cite[Corollary 4.2]{ColomboTione}) for which we need the continuity assumption. On the other hand, the proof in the second case $X\cdot Y^\perp=0$ is more delicate and relies on the geometry of the two vector fields. To be more precise, we make the observation that the condition $X\cdot Y^\perp=0$ implies that the level sets of the two Hamiltonians $H,K$ coincide, so that we can reduce the study of the commutativity of flows to individual level set, in the spirit of what was done in \cite{ABC}. In particular, the strict geometry of the two vector fields allows us to extend the result in \cite{ColomboTione} by means of purely geometric considerations, and recover the classical equivalence.

% We point out that the non-continuous case, though solvable if $X\cdot Y^\perp=0$ (see Subsection \ref{ss:non:continuous}), would require different ideas in the case $X\cdot Y^\perp=1$. We expect that, by carefully adapting the ideas in \cite{Eliodiff}, one should conclude the same results also under the non-continuity assumption.

 \subsubsection*{The Nearly incompressible case} 
In the second part of our work (Section \ref{Sect:nearly:incompr}), we assume the more general hypothesis $\dive X,\dive Y\in L^\infty(\R^2,\R)$. A key observation in our analysis is that, by classical results, if we assume that $X$, $Y$ have bounded divergence, then they are nearly incompressible in the following sense:
\begin{definition}[Nearly incompressible]
    \label{def:nearly:incomp} 
    A bounded vector field $X\in L^1_{\text{loc}}(\R^n,\R^n)\cap W^{1,p}_{\text{loc}}(\R^n,\R^n)$ is called \emph{nearly incompressible} if there exists a positive constant $C>0$ and a density $\rho:I\times\Omega\rightarrow\R$ such that $$\partial_t\rho+\dive(\rho X)=0\quad\text{in }\mathcal D'(I\times\Omega),$$ with $\rho\in[C^{-1},C]$ $\L^n$-a.e. $(t,x)\in I\times\Omega$. 
    
\end{definition}
It is clear that divergence-free vector fields are nearly incompressible. 
%The constant $C$ here is the \emph{compressibility constant}.
We recall that nearly incompressible vector fields were introduced in the study of the hyperbolic system of conservation laws named after Keyfitz and Kranzer (for a systematic treatment of this topic, we refer the reader to \cite{DeLellisNotes}). 

\bigskip

Also in this more general setting, Equation \eqref{eq:Lie:bracket:intro} provides strong geometric insights on the flows of the two vector fields. Indeed, under the hypothesis $[X,Y]=0$, the vector field
\begin{equation*}
    W\doteq \dive Y X-\dive X Y \in L^\infty
\end{equation*}
is a Hamiltonian vector field, with Lipschitz Hamiltonian given by
\begin{equation*}
    H_{W}=X\cdot Y^\perp,
\end{equation*}
where $Y^\perp$ denotes $(-Y_2,Y_1)$, namely the counterclockwise rotation of the vector field $Y$. As in the previous case of Hamiltonian vector fields, there is a clear distinction between the two behaviours occurring in the regions $\lbrace X\cdot Y^\perp=0\rbrace$ and $\lbrace X\cdot Y^\perp\not=0\rbrace$. To study the  behaviour of the flows of $X,Y$, we prove two key results. The first (Proposition \ref{thm:eq:riccardo} below) states that the flows of the two vector fields $X,Y$ remain confined in the regions where $H_W$ is bounded, namely the closed sets $\overline{\Omega}_\kappa=H_W^{-1}\left[{C}^{-1}\kappa, C\kappa\right]$. Indeed,
\begin{proposition}\label{thm:eq:riccardo}
    Let $X,Y\in W^{1,p}(\R^2,\R^2)\cap L^\infty(\R^2,\R^2)$ with bounded divergence. Moreover, assume that $[X,Y]=0$. Then the following equation holds
    \begin{equation}\label{eq:variation:det}
        X\cdot Y^\perp\circ\phi^X_t= \text{det}D\phi^X_t X\cdot Y^\perp\quad\text{a.e.},
    \end{equation}
    (and the same for the flow $\phi^Y_s$). In particular, if $X\cdot Y^\perp>0$, then there holds 
    \begin{equation*}
        \frac{1}{C} X\cdot Y^\perp\leq  X\cdot Y^\perp\circ\phi^X_t\leq C  X\cdot Y^\perp\qquad \text{a.e.,}
    \end{equation*}
for some positive constant $C>0$.
\end{proposition}
The proof of the previous proposition is a consequence of Equation \eqref{eq:Lie:bracket} below and of the boundedness of $\text{det}D\phi^X_t$, namely

\begin{equation}\label{eq:limitedness}
    \frac{1}{C}\leq \text{det}D\phi^X_t\leq C, \quad \text{a.e.}.
\end{equation}
The second result (Proposition \ref{thm:steady}) states that the two vector fields $X,Y$ are steady nearly incompressible in the sets $\Omega_\kappa$. We recall that
\begin{definition}[Steady nearly incompressible]\label{def:steady}
    A bounded vector field $X\in L^1_{\text{loc}}(\R^n;\R^n)$ is called \emph{steady nearly incompressible} in $\Omega$ if there exists a positive constant $C>0$ and a time-independent density $\rho:\Omega\rightarrow\R$ such that $$\dive(\rho X)=0\quad\text{in }\mathcal D'(\Omega),$$ with $\rho\in[C^{-1},C]$ $\L^n$-a.e. $x\in \Omega$. 
\end{definition}
Then the following Proposition holds
\begin{proposition}\label{thm:steady}
    Let $X,Y\in L^\infty(\R^2,\R^2)\cap W^{1,p}(\R^2,\R^2)$ with bounded divergence such that $[X,Y]=0$. Then for every $\kappa\in \R\setminus \lbrace{0}\rbrace$, the vector fields $X\llcorner_{\Omega_\kappa},Y\llcorner_{\Omega_\kappa}$ are steady nearly incompressible. Moreover, they have the same Lipschitz and bounded density $\rho_\kappa:\Omega_\kappa\rightarrow\R$ given by the function
    \begin{equation*}
        \rho_\kappa=\frac{1}{H_W},
    \end{equation*}
    that is 
    \begin{equation*}
        \dive(\rho_\kappa (X-Y))=0, \quad\text{in }\mathcal D'(\Omega_\kappa).
    \end{equation*}
\end{proposition}
This proposition allows us to tackle the problem of commutativity of flows separately in the two regions $\Omega=\cup _{\kappa \not=0}\Omega_\kappa$ and in $\lbrace X\cdot Y^\perp=0\rbrace,$ similarly to what was previously done in the case of Hamiltonian vector fields. In this case also, we expect that in the set $\Omega$ the flows of the two vector fields are differentiable in the spirit of \cite{Eliodiff}, while in the complementary set $\lbrace X\cdot Y^\perp=0\rbrace$ we exploit the geometry of the trajectories. We are now in a position to state the main result of this paper, namely
\begin{theorem}\label{thm:main:n:i-intro}
    Let $X,Y\in W^{1,p}(\R^2,\R^2)\cap C^0(\R^2,\R^2)$ be two bounded vector fields with $\dive X,\dive Y \in L^\infty(\R^2,\R^2).$ Then we have
\begin{equation*} \qquad [X,Y]=0 \iff \phi_t^X \circ \phi_s^Y=\phi_s^Y \circ \phi_t^X , \quad \forall t,s \in \R,\quad\text{a.e. in }\R^2.
    \end{equation*}
\end{theorem}
It is clear that Proposition \ref{thm:main:hamiltonian} follows immediately once we have established Theorem \ref{thm:main:n:i-intro}. However, we prefer to analyze the Hamiltonian case first in order to present the main ideas of the proof more clearly. Moreover, the Hamiltonian structure allows us to propose in Section \ref{Sect:hamiltonians} an alternative proof of the commutativity of the flows which relies on the analysis of level sets of Hamiltonians.

\medskip

{To summarize the above, the general idea is that, in the regions where $X\cdot Y^\perp\not=0$, one of the two vector field plays the role of the direction $\mathbf{e}$ in Theorem \ref{thm:elio} and therefore we have the differentiability property of the two flows. At the points where the direction $\mathbf{e}$ cannot be found, we exploit the fact that the trajectories of the two vector fields are the same. In particular, we remark that the non-symmetric assumption `$\phi^X_t$ is weakly Lie differentiable w.r.t. $Y$' (without asking anything on the flow of the vector field $Y$) is symmetric, since in the same region $X\cdot Y^\perp\not=0$ the flow of the vector field $Y$ is weakly Lie differentiable w.r.t. $X$.}
\medskip

We conclude this paragraph by emphasizing that, since $[X,Y]=0$ imposes strict geometric conditions on the flows of the two vector fields, we expect that also in higher dimension the geometry of trajectories could be analyzed to improve the result in \cite{ColomboTione}, since analogous of Equation \eqref{eq:Lie:bracket:intro} are available. On the other hand, it is unclear whether in higher dimension we can still expect Sobolev regularity results for the flows under study. Equation \eqref{eq:Lie:bracket:intro} turned out to be a surprisingly rich source of information about the geometric properties of the two vector fields, and we expect that it could be further investigated for our understanding of the underlying geometric structure given by the vanishing Lie bracket. 

\subsubsection*{Extension to $\BV$ setting}
%We briefly comment on the extension to the $\BV$ setting, 
In the final part of this work, Section \ref{Sect:BV}, we briefly discuss how to generalize the previous results to $\BV$ vector fields, since the results of our earlier discussions hold also in this case. In our opinion, the extension to the $\BV$ setting is not trivial. First and foremost, it requires the following distributional definition of Lie Bracket.

\begin{definition}[Distributional Lie bracket]\label{def:Lie:bracket}
    Let $X$, $Y \in \BV_{\text{loc}}(\R^n,\R^n)\cap L^\infty(\R^n,\R^n)$ be vector fields with bounded divergence. Then we define the (distributional) Lie bracket $[X,Y]$ by setting, for every $\varphi\in C^1_c(\R^n,\R^n)$,
\begin{equation}\label{eq:def:lie:brack}
    \int_{\R^n} ([X,Y],\varphi) dz \doteq\int_{\R^n}\left((X,D\varphi Y)+\dive Y(X,\varphi)-(Y,D\varphi X)-\dive X(Y,\varphi) \right)dz.
\end{equation}
\end{definition}
We remark that the Lie bracket in \eqref{eq:def:lie:brack} is well-defined since we assume that the vector fields $X$ and $Y$ have bounded divergence. Moreover, we emphasize that our definition of Lie bracket is consistent with the classical one if $X,Y$ are smooth vector fields and reduces to the one in \cite{ColomboTione} if we identify $X$ and $Y$ with their precise representative, by Vol'pert chain rule formula (see Sections \ref{S:preliminaries} and \ref{Sect:BV}). Owing to Definition \ref{def:Lie:bracket}, we prove the following proposition, which extends to the $\BV$ setting one of the implications in \eqref{eq:commutativity-classical}.

\begin{proposition}\label{thm:comm-BV}
 Let $X$, $Y \in \BV_{\text{loc}}(\R^n,\R^n)\cap L^\infty(\R^n,\R^n)$ be vector fields with bounded divergence and let $\phi_t^X$, $\phi_s^Y$ be their Regular Lagrangian Flows. Then the following implication holds
 \begin{equation}
   \phi_t^X \circ \phi_s^Y=\phi_s^Y \circ \phi_t^X , \quad \forall t,s \in \R,\quad\text{a.e. in }\R^n \quad \Rightarrow \quad [X,Y]=0.
 \end{equation}
\end{proposition}
The proof of Proposition \ref{thm:comm-BV} is obtained by adapting the method used in \cite{ColomboTione}. In contrast to the Sobolev case studied in \cite{ColomboTione}, it seems to be not straightforward to prove the missing implication in Proposition \ref{thm:comm-BV}. For instance, a key step to show that the vanishing Lie bracket implies that the two flows commute, is to evaluate the function
\begin{equation*}
    t\rightarrow [X,Y]\circ\phi^X_t,
\end{equation*}
and to apply the hypothesis on vanishing Lie Bracket on the trajectories. In the $\BV$ setting, where we need a distributional definition of Lie Bracket, it is unclear to us how to give meaning to the previous function. For this reason, the geometric approach we follow in Sections \ref{Sect:hamiltonians} and \ref{Sect:nearly:incompr} can not be straightforwardly generalized to the $\BV$ setting (it is in particular not possible to employ Theorem \ref{thm:Marconi} to deal with the case $X \cdot Y^\perp \neq 0$). This constitutes a key difference between the Sobolev and the $\BV$ case and therefore we leave the reverse implication as an open problem.

\subsubsection{Outline of the paper}
The structure of the paper is as follows. Section \ref{S:preliminaries} is devoted to preliminaries: we collect some known results regarding $\BV$ functions, we introduce the notion of regular Lagrangian flow and we recall the structure of level
sets of Lipschitz functions with compact support and gradient with bounded variation. In Section \ref{Sect:hamiltonians}, we prove Proposition \ref{thm:main:hamiltonian} through a dimension-reduction argument in the spirit of \cite{ABC}. In Section \ref{Sect:nearly:incompr} we generalize the previous arguments to the case of rough vector fields with bounded divergence and we prove our main result, Theorem \ref{thm:main:n:i-intro}. Finally, some comments on the extension to the case of $\BV$ vector fields are presented in Section \ref{Sect:BV}.

%This is a key difference between the Sobolev and the $\BV$ case is that one should be able to define the distributional Lie Bracket along the flow of one of the two vector fields, since it is an essential features of the proof of \cite{ColomboTione}. \\

\bigskip

 \textbf{Acknowledgements.} The authors acknowledge Riccardo Tione for suggesting the problem and Riccardo Tione, Giacomo del Nin and Elio Marconi for many useful discussions and insights on the topic.

%\subsubsection{Remarks on the weakly Lie differentiability.}

%We will see in the following that the additional assumption 
%\begin{equation}\label{eq:add:assumpt}
 %   \phi^X_t \text{ is weakly (Lie) differentiable in direction $Y$ with locally bounded derivative}
%\end{equation}
%is essentially a directional derivative of the flow of the vector field $X$ in the direction of the vector field $Y$. In particular, if $X,Y$ are \emph{linearly independent} (see Definition..), since $\phi^X$ is differentiable w.r.t. $X$ (property encoded in the fact that $[X,X]=0$ trivially), heuristically this condition is  telling us that $\phi^X_t$ is \emph{differentiable}.
%We remark also that, despite of the fact this condition may not seem symmetric in the two vector fields,  the statement of [Colombo, Tione] gives that, if $[X,Y]=0$ and $\phi^X_t$ is weakly Lie differentiable w.r.t. $Y$, then also $\phi^Y_t$ is weakly Lie differentiable w.r.t. $X$: in particular, the search for counterexamples that disprove the necessity of this condition has to be performed, according to us, in the class of vector fields whose flow is not differentiable (see [Marconi,Jabin]). 

\section{Preliminaries}\label{S:preliminaries}
In this section, we collect some preliminary results that will be used throughout the paper. We first give a quick overview of $\BV$ functions, Regular Lagrangian Flows and Commutator estimates. Finally, in the last subsection, we give a characterization of Level sets of Lipschitz functions with compact support.  
\subsection{BV functions}
\label{Ss:BV_funct}
In this subsection, we recall some known results concerning functions of bounded variation. For a more comprehensive treatment of the topic, we refer the reader to \cite{AFP}. Let  $u\in\BV(\Omega,\R^n)$ and $Du\in\M_b(\Omega)^{n\times n}$ the $n\times n$-valued measure representing its distributional derivative. We recall the canonical decomposition of the measure $Du$ %with respect to  the Lebesgue measure $\L^n$ we find
\begin{equation*}
%\label{leb:dec}
    Du=D^{\textrm{a}}u+D^\textrm{s}u = D^\textrm{a} u + D^\mathrm{c} u + D^\mathrm{j} u,
\end{equation*}
where $D^{\textrm{a}}u$ is the absolutely continuous part of $Du$ with respect to the $n$-dimensional Lebesgue measure $\L^n$ and $D^\textrm{s}u$ is the singular part of $Du$ with respect to $\L^n$. Moreover, we split $D^\textrm{s}u$ into the jump part $D^s u$ and the Cantor part $D^\mathrm{c} u$. Since $D^a u $ is an absolutely continuous function, we can write it as $D^a u = \nabla u \L^n$, i.e. we denote by $\nabla u =(\partial_i u^j) \in \R^{n \times n}$ the density of $D^a u $ with respect to $\L^n$.
We finally recall a structure theorem for $\BV$ vector fields from \cite{AFP}.

\begin{theorem}\label{thm:structure:BV}
Let $u \in \BV(\R^n,\R^n)$. Then there exists a decomposition of $\R^n=S_u \cup J_u \cup D_u$ such that the following properties hold true:
\begin{itemize}
    \item[1.] $\mathcal{H}^{n-1}(S_u)=0$;
    \item[2.] $J_u$ is a Borel subset of $S_u$ and there exist Borel functions $\nu:J_u \rightarrow \mathbb{S}^{n-1}$, $u^-$, $u^+:J_u \rightarrow \R^n$ such that, for every $\overline{x} \in J_u$, we have
    \begin{equation*}
    \dashint_{B_r^-(\overline{x})}  |u(x)-u^-(\overline{x})|dx  =o(1), \quad  \dashint_{B_r^+(\overline{x})}  |u(x)-u^+(\overline{x})|dx  =o(1), \quad \text{as $r \to 0$},
    \end{equation*}
    where $B_r^-(\overline{x}) \doteq \lbrace x \in B_r: (x-\overline{x},\nu)<0\rbrace$ and $B_r^+(\overline{x}) \doteq \lbrace x \in B_r: (x-\overline{x},\nu)>0\rbrace$;
    \item[3.] for every $\overline{x} \in D_u$, there exists $\tilde{u}(\overline{x})$ such that
    \begin{equation*}
    \dashint_{B_r(\overline{x})}  |u(x)-\tilde{u}(\overline{x})|dx  =o(1),  \quad \text{as $r \to 0$}.
    \end{equation*}
\end{itemize}
\end{theorem}
Taking advantage of the previous theorem, it is possible to assign a value to $u$ at every point $x \notin S_u \setminus J_u$. More precisely, the precise representative of $u \in \BV(\Omega,\R^n)$ is the Borel function $u^*$ such that $u^*(x)=\tilde{u}(x)$ if $x \in D_u$ and $u^*(x)=(u^+(x)+u^-(x))/2$ if $x \in J_u$. The precise representative allows us to state the general Leibniz rule for the
product of two $\BV$ functions. More precisely, if $u$, $v \in \BV(\R^n,\R^n) \cap L^\infty(\R^n,\R^n)$, then $u v \in \BV(\R^n,\R^n)$ with
\begin{equation}\label{eq:leibniz-rule}
D    (uv)=v^* Du+ u^*Dv,
\end{equation}
see \cite{Volpert} or \cite[Theorem 3.96]{AFP}.

\subsection{Regular Lagrangian Flows}\label{sect:RLF}
Throughout this paper, we will consider autonomous vector fields $X,Y:\R^n\rightarrow \R^n$ in the space $\BV_{\text{loc}}(\R^n,\R^n)\cap L^\infty(\R^n,\R^n)$ (in short $X,Y\in \BV_{\text{loc}}\cap L^\infty$) with $\dive X,\dive Y\in L^\infty(\R^n,\R)$. We first recall that, when the velocity fields $X,Y$ are Lipschitz, then their \emph{flows} are well-defined in the classical sense, i.e. they are the maps $\phi^X_t,\phi^Y_s:\R\times \R^n\rightarrow \R^n$ satisfying
\begin{equation}
 \label{ODE}
    \begin{cases}
   \partial_t\phi^X_t(x)=X(\phi^X_t(x)), & \\
    \phi^X_0(x)=x,
    \end{cases}
\end{equation}
(and same for $\phi^Y_s$). When dealing with rough vector fields, the notion of pointwise uniqueness of \eqref{ODE} is not any more the appropriate one and we only have uniqueness ``in the selection sense", as encoded in the following definition.
\begin{definition}
  \label{def:RLF}
    Let $X\in \BV_{\text{loc}}\cap L^\infty$. A map $\phi^X_t:\R\times \R^n\rightarrow \R^n$ is a \emph{Regular Lagrangian Flow} (RLF) for the vector field $X$ if
    \begin{enumerate}
        \item the following system holds in the sense of distributions
        \begin{equation}\label{eq:Cauchy:RLF}
            \begin{cases}
\partial_t \phi_t^X(x)=X \circ \phi_t^X(x); & \\
    \phi_0^X(x)=x.
    \end{cases}
        \end{equation}
        \item 
       there exists a positive constant $C$, independent of $t$, such that for every $t \in \R$, we have that $| \lbrace x: \phi_t^X(x) \in A \rbrace|\leq C|A|$ for every Borel set $A$.
       % \begin{equation}
        %  \label{comp:cond}
         %   \L^2(\phi_t^{-1}(A))\leq C\L^2(A),\quad\forall A\in\mathcal{B}(\R^2).
        %\end{equation}
    \end{enumerate}
\end{definition}
We observe that the vector fields we consider in this work are autonomous, even if the theory is developed also for non-autonomous vector fields, i.e. in general $X=X(t,x)$. More precisely, we recall that Regular Lagrangian flows were introduced in \cite{DiPerna:Lions}, where the authors prove their existence and uniqueness for vector fields $X \in W^{1,p}(\R^n,\R^n)$, $p \geq 1$, with bounded divergence. The theory was subsequently extended to $\BV$ vector fields with bounded divergence in \cite{Ambrosio:BV}. Finally, uniqueness results in the more general class of nearly incompressible vector fields were established in \cite{BianchiniBonicatto}. Here we briefly recall the stability result for RLFs of $\BV$ vector fields, that is 
\begin{theorem}[Stability, Theorem 6.3 \cite{Ambrosio:Luminy}] 
\label{Theo:stability}
		Let   $X_n,X\in \BV_{\text{loc}}\cap L^\infty$ with $\dive (X_n)$ with bounded divergence and let $\phi^{X^n},\phi^X$
 be their corresponding Regular Lagrangian Flows. Assume that 
		\begin{equation*}
			%\label{stab:cond}
			||X_n-X||_{L^1}\to 0\quad\text{as }n\to\infty,\quad \|\dive X_n\|_\infty\leq C,\quad \forall n,
		\end{equation*}
		then
		\begin{equation*}
			%\label{stability}
			\lim_{n\to\infty}\int_{\R^n} \sup_{t} |\phi^{X^n}_t(x)-\phi^X_t(x)|dx =0.
		\end{equation*}
	\end{theorem}
	
From now on we denote by $\phi_t^X(\cdot)_\sharp \mathcal{L}^d$ the push-forward of the Lebesgue measure on $\R^d$ via the function $\phi_t^X(\cdot)$. We observe, as already pointed out in \cite{DeLellisNotes}, that Condition $2.$ in Definition \ref{def:RLF} is equivalent to
\begin{equation*}
    \phi_t^X(\cdot)_\sharp \mathcal{L}^d \ll \mathcal{L}^d
    \end{equation*}
    for every $t \geq 0$. Then, for every RLF, there exists a $\xi\in L^\infty([-T,T]\times\R^n)$, for every $T>0$, such that $ \phi_t^X(\cdot)_\sharp \mathcal{L}^d= \xi \mathcal{L}^d$. The function $\xi$ will be called \textit{density} of the flow $\phi_t^X$ and, by definition, it satisfies the following “change of variables” identity
 \begin{equation}\label{eq:change-of-var}
     \int_{\R^{n+1}} \varphi(t,\phi^X_{-t}(z))dzdt= \int_{\R^{n+1}} \varphi(t,z)\xi(t,z)dzdt,\quad\forall\varphi\in\mathcal{C}_c(\R^{n+1}).
 \end{equation}   
%We observe that, as already pointed out in \cite{ColomboTione}, Condition \eqref{comp:cond} guarantees the existence 
In the smooth setting, clearly $\xi=\text{det}D\phi^X_t.$ Moreover, $\xi\in L^\infty([-T,T]\times\R^n)$ for every $T>0$ and it satisfies the analogous of Liouville Theorem, that is
\begin{equation}
    \label{eq:Liouville}
    \partial_t \xi=\dive(X)\circ\phi^X_t\xi\quad\text{distributionally}.
\end{equation}
See \cite{ColomboTione} for the details, where the Stability Theorem is used to obtain this result. 
\subsection{Structure of level sets of $\text{Lip}_c$ functions}
In this subsection, we consider $X\in L^\infty(\R^2;\R^2)\cap \BV_{\text{loc}}(\R^2)$ with $\dive X=0$ a.e. and we additionally assume that $X$ has compact support (this is not restrictive since we consider bounded vector fields). Then there exists $H\in \text{Lip}_c(\R^2,\R)$ such that $ X=\nabla^\perp H$, where $\nabla^\perp=(-\partial_2,\partial_1)$ is the orthogonal gradient.  Before we give the characterization of the level sets of $H$ (the results we state here are from \cite{BourgainKK} as in \cite{Marconi_Bonicatto_poly}), we recall that a curve $\Gamma$ in $\R^2$ is the image of a continuous, non-constant $\gamma:[a,b]\rightarrow\R^2$, thus it is a compact connected set. We say that $\Gamma$ is a simple closed curve if there exists a parametrization $\gamma:[a,b]\rightarrow \R^2$, injective on $[a,b)$ such that $\gamma(a)=\gamma(b)$.
\begin{theorem}\label{thm:struct:level:sets}
We consider $X\in L^\infty(\R^2;\R^2)\cap \BV_{\text{loc}}(\R^2)$ with $\dive X=0$ a.e. and with compact support. Then there exists a set $C\subset H(\R^2)$ with $\L^1(H(\R^2)\setminus C)=0$ and such that for every $h\in C$ the following properties hold:
    \begin{itemize}
        \item $\H^1(H^{-1}(h))<\infty$;
        \item  there exists $N=N(h)\in \N$ and $C_i(h)$ with $i\in\lbrace 1,2,\dots, N\rbrace$ disjoint simple closed curves such that
        \begin{equation*}
            H^{-1}(h)=\cup_{i=1}^N C_i(h).
        \end{equation*} 
        Furthermore, there exist Lipschitz parametrizations $\gamma_i(h):[0,\ell_i(h)]\rightarrow\R^2$ with $\left|\frac{d}{dt} ({\gamma_i(h)})_t\right|=1$ $\L^1$-a.e. $t$ and $\gamma_i(h)([0,\ell_i(h)])=C_i(h).$ \end{itemize}
        Moreover, for every $\epsilon>0$ there exists $C_\epsilon\subset H(\R^2)$, and a constant $c>0$ such that $\L^1(H(\R^2)\setminus C_\epsilon)<\epsilon$ and for every $h\in C_\epsilon$ it holds
        \begin{equation*}
            \forall x\in H^{-1}(h)\cap D_X,\qquad |X(x)|\geq c.
        \end{equation*}
   
\end{theorem}

\begin{corollary}\label{cor:pos:vect:level}
    In the assumption of the previous theorem, there exists a set $\tilde E_X\subset H(\R^2)$ with $\L^1(H(\R^2)\setminus \tilde E_X)=0$ such that for every $h\in \tilde E_X$, there exists $c_h>0$ such that $|X(x)|>c_h$ for every $x\in\lbrace H=h\rbrace $.
\end{corollary}
\begin{remark}
    We remark that a slightly different characterization is needed for vector fields that do not have compact support, where we lose the property for the level sets to be finite union of simple closed curves. In this paper we omit this description in order to avoid further technicalities.
\end{remark}
We recall that, under this assumption, the Hamiltonian structure is compatible with the flow, namely 
\begin{theorem}[Theorem 2.6 in \cite{Marconi_Bonicatto_poly}]
    Let $X=\nabla^\perp H$ be a $\BV$ vector field. If $\phi^X_t:\R^2\rightarrow\R^2$ is its unique RLF, then for $\L^2$-a.e. $x\in\R^2$ and for every $t\in\R$,
    \begin{equation*}
        H(\phi^X_t(x))=H(x).
    \end{equation*}
\end{theorem}
In particular, if we fix a representative for the RLF of $X$, it holds that for every $h\in E_X$, for every $x\in C_i(h)$, where $C_i(h)\in H^{-1}(h)$ is some cycle, and $\gamma_i(h)$ is the parametrization of Theorem \ref{thm:struct:level:sets}, then $\phi^X_t(x)=\gamma_i(h)(\bar\tau)$, where $\bar\tau\in[0,\ell_i(h)]$ is uniquely determined by
\begin{equation}\label{eq:ham:vect:field}
    t=N \int_0^{\ell_i(h)}\frac{1}{|X(\gamma_i(h)(\tau))|}d\tau+\int_{\gamma_i(h)^{-1}(x)}^{\bar\tau}\frac{1}{|X(\gamma_i(h)(\tau))|}d\tau,
\end{equation}
where $N\in\N$.

\section{Hamiltonian vector fields}\label{Sect:hamiltonians}
In this section, we consider two autonomous divergence-free vector fields in the real plane. The analysis in this setting is facilitated by the Hamiltonian structure of the two vector fields. Indeed, if $X,Y\in L^\infty(\R^2;\R^2)$ with $\dive X=\dive Y=0$ a.e., then there exist two Lipschitz Hamiltonians $H,K:\R^2\rightarrow \R$ such that
\begin{equation}
    X=\nabla^\perp H,\qquad Y=\nabla^\perp K,
\end{equation}
where $\nabla^\perp=(-\partial_2,\partial_1)$ is the orthogonal gradient. Here the main property that we want to exploit, requiring the Sobolev regularity of the vector fields, is the fact that the trajectories of particles advected by $X$ (resp. $Y$) are contained in the level sets of the Hamiltonian $H$ (resp. $K$). Indeed, if $X,Y$ have $W^{1,p}_\text{loc}$ regularity in space, the Hamiltonians are preserved by their RLFs (see Definition \ref{def:RLF}), namely 
\begin{equation*}
    H\circ \phi^X_t=\text{const}, \qquad K\circ \phi^Y_s=\text{const}, 
\end{equation*}
and the RLFs are uniquely determined by
\begin{equation*}
    \phi^X_t(x)=x,\quad\forall x\in H^{-1}(S_H),
\end{equation*}
where $S_H$ denotes the singular values of $H$, namely
\begin{equation*}
    S_H=\lbrace h\in \R: \exists x \in \R^2: H(x)=h, X(x)=0\rbrace,
\end{equation*}
which is a closed set, under the assumption that $X$ is continuous with compact support. Clearly, the same holds for the Hamiltonian $K$ along the flow of the vector field $Y$ (under the same continuity assumption on the vector field $Y$).  Moreover, the flow $\phi^X_t$ (resp. $\phi_s^Y$) is the unique solution of $\partial_t\phi^X_t=X(\phi^X_t)$ (resp. $\partial_s\phi^Y_t=Y(\phi^Y_s)$), understood as a 1-dimensional problem on the level sets of the Hamiltonian $H$, see in particular Formula \eqref{eq:ham:vect:field}.

\medskip

 The aim of this section is to prove the following proposition.
\begin{proposition}\label{thm:commutativity}
    Let $X,Y\in W^{1,p}_{loc}(\R^2,\R^2)\cap C^0(\R^2,\R^2)$ be two bounded vector fields and assume that $\dive X=\dive Y=0$ almost everywhere. Then 
    \begin{equation}
        [X,Y]=0 \quad \Rightarrow \quad \phi^X_t\circ\phi^Y_s=\phi^Y_s\circ\phi^X_t,\quad\forall t,s\in\R, \quad \L^2\text{-a.e. } x\in\R^2.
    \end{equation}
\end{proposition}
This result extends in particular \cite[ Corollary 4.2]{ColomboTione} to vector fields that can vanish on the real plane, so that the classical equivalence is established. As explained above, the main property used here is the following well-known fact in Hamiltonian dynamics
\begin{equation}\label{eq:condition-possion}
    [X,Y]=0 \quad\iff \lbrace H,K\rbrace=\text{const},
\end{equation}
where $\lbrace \cdot,\cdot\rbrace$ denotes the Poisson Bracket. Such formula holds also in this weak context, as proved in Lemma \ref{lem:eq:for:bracket}. In the two-dimensional case, condition \eqref{eq:condition-possion} is equivalent to the following property of the two vector fields:
\begin{equation*}
    X\cdot Y^\perp=\text{const}.
\end{equation*}

%\begin{remark}
 %   We point out that the condition $ X\cdot Y^\perp=\text{const}$ can be written as $\text{det}(\nabla K,\nabla H)=\text{const},$ connecting this problem to differential inclusions.
%\end{remark}
Before proving Proposition \ref{thm:commutativity}, we give the following general result
\begin{lemma}\label{lem:eq:for:bracket}
    Let $X,Y\in W^{1,p}_{\text{loc}}(\R^2)\cap L^\infty(\R^2)$ with $\dive X,\dive Y\in L^\infty(\R^2)$. Then
    \begin{equation}\label{eq:Lie:bracket}
        [X,Y]= \dive YX-\dive X Y -\nabla^\perp (X\cdot Y^\perp), \quad\text{distributionally}
        \end{equation}
        In particular, if $\dive X=\dive Y=0$ then 
        \begin{equation*}
            [X,Y]=0\qquad\text{iff}\qquad X\cdot Y^\perp=\text{const}.
        \end{equation*}
\end{lemma}
\begin{proof} 
Let $\varphi\in C^1_c(\R^2)$. Then, by definition (see Equation \eqref{eq:def:lie:brack}), the distributional Lie Bracket reads as
\begin{align*}
    \int_{\R^2} ([X,Y],\varphi)dz&=\int_{\R^2}(X,D\varphi Y)-(Y,D\varphi X)+\dive Y(X,\varphi)-\dive X(Y,\varphi)dz.
\end{align*}

In the general case, we have
\begin{align*}
    \int_{\R^2}(X,D\varphi Y)-(Y,D\varphi X)dz&=\int_{\R^2}(Y_1X_1\partial_1\varphi_1+X_1Y_2\partial_2\varphi_1+ Y_1X_2\partial_1\varphi_2+Y_2X_2\partial_2\varphi_2)dz\\ &-\int_{\R^2}(X_1Y_1\partial_1\varphi_1+Y_1X_2\partial_2\varphi_1+ X_1Y_2\partial_1\varphi_2+X_2Y_2\partial_2\varphi_2)dz\\&
    =\int_{\R^2}(X_1Y_2\partial_2\varphi_1+ Y_1X_2\partial_1\varphi_2-Y_1X_2\partial_2\varphi_1- X_1Y_2\partial_1\varphi_2)dz\\&
    =\int_{\R^2}(X\cdot Y^\perp)(-\partial_2\varphi_1+\partial_1\varphi_1)dz,
\end{align*}
which immediately gives the statement.
\end{proof}
\begin{remark}
    We observe that Equation \eqref{eq:Lie:bracket} and Lemma \ref{lem:eq:for:bracket} hold also in the $\BV$ setting.
\end{remark}
\begin{remark}
   As anticipated before, the condition $[X,Y]=0$ for Hamiltonian vector fields corresponds to $\lbrace H,K\rbrace=\text{const}$. In particular, this formula holds in any even dimension, giving strong geometric properties to the flows of Hamiltonian vector fields (that is, the Lie bracket of two Hamiltonian vector fields is itself a Hamiltonian vector field). 
\end{remark}

\subsection{Analysis of $X\cdot Y^\perp=\text{const}$}\label{sub:scalprod}
This subsection is devoted to the proof of Proposition \ref{thm:commutativity}, which in turn implies Proposition \ref{thm:main:hamiltonian} in virtue of \cite[Theorem 1.1]{ColomboTione}. Thus, we are just left with proving that the vanishing of the Lie bracket implies the commutativity of flows in the case where the two vector fields $X,Y$ are Hamiltonian and continuous with Sobolev regularity. As explained above, this boils down to discussing the two separate cases $X\cdot Y^\perp\not=0$ and $X\cdot Y^\perp=0$. In particular, if $X\cdot Y^\perp\not=0$, under the assumption of continuity of the two vector fields, the commutativity is a corollary of the result in \cite{Eliodiff}, that we recall here for the reader's convenience as stated in \cite{ColomboTione}, namely
\begin{theorem}[\cite{Eliodiff}]\label{thm:elio}
    Let $X\in W^{1,p}_{\text{loc}}(\R^2,\R^2)\cap L^\infty(\R^2,\R^2)$ be a continuous vector field with zero divergence. If $X\not=0$, then its RLF $\phi^X_t\in L^\infty_{\text{loc}}(\R,W^{1,p}_{\text{loc}}(\R^2,\R^2)).$
\end{theorem}
We remark that the continuity of the vector field is necessary to ensure that $X$ is uniformly bounded from below. In the case $X\cdot Y^\perp =1$, the assumptions of the previous theorem are trivially satisfied, as condition
$X\cdot Y^\perp=1$  implies that there exists a positive constant $c>0$ such that $|X|\geq c$. Indeed, 
\begin{equation*}
    1\leq|X||Y|\quad \longrightarrow \quad |X|\geq \frac{1}{\|Y\|_\infty},
\end{equation*}
and therefore Theorem \ref{thm:elio} guarantees the differentiability of the flow $\phi^X_t$, which in turn implies the weak Lie differentiability (Definition \ref{def:Lie:derivative}).\\

According to the previous arguments, from now on we will analyze the more delicate case $X\cdot Y^\perp=0$ everywhere. In this case, thanks to the boundedness of the vector fields, we can assume without loss of generality that $X,Y$ have compact support, in order to make the analysis more transparent. 

We first recall that $S_H,S_K$ are the sets of singular values of $H,K$ respectively, while $R_H,R_K$, defined as $R_H=H(\R^2)\setminus S_H$, $R_K=K(\R^2)\setminus S_K$, are the sets of regular values. The identity $X\cdot Y^\perp=0$, at the level of Hamiltonians of the two vector fields, translates into the property that the level sets of the two Hamiltonians $H,K$ coincide. This has to be intended in the following sense: consider a simple closed curve $C_i(h)\subset \lbrace H=h\rbrace$ with $h\in \tilde E_X\subset R_H$ (recall the definition of the set $\tilde E_X$ from Corollary \ref{cor:pos:vect:level}), and assume that $Y(x)\not=0$ for $x\in \lbrace H=h\rbrace$. Then there exists $k\in \tilde E_Y$ such that $C_i(h)\subset \lbrace K=k\rbrace.$ This crucial property of the level sets of the Hamiltonians will be clarified even further in the proofs of the upcoming propositions. \\

With these considerations at hand, we aim at giving a partition of the space $\R^2$ into sets where the commutativity of flows can be proved independently and with different ideas, up to $\L^2$-negligible sets. Before starting our analysis, we recall that:
\begin{equation}\label{eq:ident}
\phi^X_t(x)=x,\quad\forall x\in H^{-1}(S_H), \quad \phi^Y_s(x)=x, \quad \forall x\in K^{-1}(S_K).
\end{equation}

Moreover, the RLFs $\phi^X_t$, $\phi^Y_s$ are integral solution of the Cauchy Problem \eqref{eq:Cauchy:RLF} for a.e. initial data $x\in A\subset\R^2$, $\L^2(A^c)=0$. For our analysis, we will limit the study to the set $A$. We recall that in this work the commutativity of flows has to be intended in the following weak sense: we say that the flows $\phi_t^X$ and $\phi_s^Y$ commute if, for every $s,t$ fixed, there exists $\Omega_{s,t}$ such that $\L^2(\Omega_{s,t}^c)=0$ and the flows commute in $\Omega_{s,t}$. However, in this section, we rely on a stronger notion of commutativity, namely we we prove that there exists a set $\Omega$ with $\L^2(\Omega ^c)=0$ such that
\begin{equation}\label{eq:strong:commutativity}
    \phi^X_t\circ\phi^Y_s=\phi^Y_s\circ\phi^X_t,\quad\forall t,s\in\R, \forall x\in \Omega.
\end{equation}
We recall the following property of the Structure Theorem \ref{thm:struct:level:sets}:
\begin{equation*}
    \exists \tilde E_X\subset H(\R^2), \text{ with } \L^1(H(\R^2)\setminus \tilde E_X)=0:\quad\forall h\in\tilde E_X, X(x)\not=0,\quad\forall x\in \lbrace H=h\rbrace.
\end{equation*}
Moreover, for such $h$, $\lbrace H=h\rbrace$ is finite union of simple closed curves $C_i(h)$, which are the connected components of the level set $\lbrace H=h\rbrace$. \\

We are now in a position to start our analysis.
We partition the space $\R^2$ as follows:
    \begin{equation*}
        \R^2=\left(\lbrace X\not=0\rbrace\cap E_X\right) \cup \left(\lbrace X\not=0\rbrace\cap E_X^c\right) \cup \lbrace X=0\rbrace,
    \end{equation*}
    where the set $E_X=H^{-1}(\tilde E_X)$. We define in the same way $\tilde E_Y$ and $E_Y$. By Coarea formula, it holds that $\L^2(\lbrace X\not =0\rbrace \cap E_X^c)=0$, since
    \begin{equation}\label{eq:coarea-form}
        \int_{\lbrace X\not =0\rbrace \cap E_X^c}|\nabla H|dx=\int_\R \H^1(\lbrace X\not =0\rbrace \cap E_X^c \cap H^{-1}(t))dt.
    \end{equation}
    The RHS is clearly zero by definition of $E_X$ and $|\nabla H|\not=0$ in the LHS since $\nabla^\perp H=X$.
Owing to this property, we give the following
\begin{corollary}
    The space $\R^2$ can be partitioned as follows
    \begin{equation}\label{eq:decom-R2}
        \left[\left(\lbrace X\not=0\rbrace \cap E_X\right)\cap \left(\lbrace Y\not=0\rbrace \cap E_Y\right)\right]\cup\lbrace X=0\rbrace \cup\lbrace Y=0\rbrace \cup N,
    \end{equation}
    where $\L^2(N)=0$ and $E_X,E_Y$ are the spaces defined above.
\end{corollary}
\begin{proof}
The proof follows easily from set-theoretical considerations and \eqref{eq:coarea-form}.
\end{proof}
In virtue of the previous corollary, we are left with proving separately the commutativity of the flows on the two non-negligible sets in \eqref{eq:decom-R2}, as done in the two upcoming propositions.
\begin{proposition}\label{prop:zero}
     In the set $\lbrace X=0\rbrace  \cup \lbrace Y=0\rbrace$ the flows commute, in the sense that there exists $\Omega'\subset \lbrace X=0\rbrace  \cup \lbrace Y=0\rbrace$ with $\L^2(\Omega'^c)=0$ such that
     \begin{equation*}
    \phi^X_t\circ\phi^Y_s=\phi^Y_s\circ\phi^X_t,\quad\forall t,s\in\R, \forall x\in \Omega'.
\end{equation*}
\end{proposition}
   
\begin{proof}
    We fix $x\in \lbrace X=0\rbrace  \cup \lbrace Y=0\rbrace$ and w.l.o.g. we can assume $X(x)=0$. If $Y(x)=0$, then the flows $\phi^X_t,\phi^Y_s$ clearly commute for every $t,s$ being the flow maps the identity. Hence, we assume that $X(x)=0$ and $Y(x)\not=0$. By the same reasoning as in \eqref{eq:coarea-form}, we first observe that it holds that $\L^2(\lbrace x: X(x)=0, Y(x)\not=0, x\in E_Y^c\rbrace )=0$, by the monotonicity of the Lebesgue measure.
    
   Thus, we are left with analyzing the case $X(x)=0, Y(x)\not=0$, and $x\in E_Y$. Since $x \in E_Y$, by definition of $E_Y$, we have that $Y(y)\not=0$ for every $y\in\lbrace K=K(x)\rbrace$. We now denote by $\mathcal{C}_k$ the connected component of $\lbrace K=K(x)\rbrace$ containing both $x$ and $y$ (we recall that $\mathcal{C}_k$ is a simple closed curve by Theorem \ref{thm:struct:level:sets}).
     If $X(y)=0$ for every $y\in \mathcal C_k$, then in particular $X\circ \phi^Y_s(x)=0$ for every $s\in \R$, thus the commutativity of flows would follow immediately. Observe that there exists $h\in H(\R^2)$ such that $\mathcal C_k\subset \lbrace H=h\rbrace,$ since $X\cdot Y^\perp=0$.
    As a consequence, if there exists $y\in \mathcal C_k$ such that $X(y)\not=0$, then it must be that $y\in \lbrace X\not=0\rbrace\cap E_X^c$ (if not, we would have $X(x)\not=0$ by definition of $E_X$). Observe that
    \begin{equation*}
        \L^2(E)=\mathcal L^2(\lbrace y \in H^{-1}(h): \exists x: X(x)=0, X(y)\not=0, \exists
        \mathcal{C}_k\quad\text{connected component: }x,y\in \mathcal{C}_k\rbrace)=0,
    \end{equation*}
    since $E$ is a subset of $\lbrace X\not=0\rbrace \cap E_X^c$, which we proved to be $\L^2$-negligible. Notice that, since $X$ is continuous and $X(y)\not=0$, for the connected component $\mathcal C_k$ satisfying the previous property, there exists $\gamma_k:[-T_k,T_k]\rightarrow \R^2$ with $\gamma_k([-T_k,T_k]) \subset \mathcal{C}_k$ and $\gamma_k(0)=y$ such that $X(\gamma_k(t))\not=0$ for all $t\in[-T_k,T_k]$. In terms of flows, this means that there exists $q\in\Q$ rational number such that $X(\phi^Y_q(x))\not=0$. Therefore
    \begin{equation*}
        \lbrace x: X(x)=0, Y(x)\not=0, x\in E_Y\rbrace \subset\cup_{q\in \Q} (\phi^Y_q)^{-1}(E).
    \end{equation*}
    Since $Y$ is incompressible and $\L^2(E)=0$, this immediately gives that $\L^2(\lbrace x: X(x)=0, Y(x)\not=0, x\in E_Y\rbrace)=0$, which concludes the proof.

\end{proof}
\begin{remark}\label{rmk:gronwall}
    Since the previous proof only relies on geometric considerations on level sets and does not use the regularity of the vector fields $X,Y$, it is in particular valid also in the $\BV$ setting. However, we remark that in the case of Sobolev regularity, the previous result is an easy corollary of Gronwall's lemma. Indeed, we have
    \begin{gather*}
        \frac{d}{ds} |X\circ\phi^Y_s|^2={2}\langle X\circ\phi^Y_s, DX(\phi^Y_s)Y\circ\phi^Y_s\rangle={2}\langle X\circ\phi^Y_s, DY(\phi^Y_s)X\circ\phi^Y_s\rangle\leq {2}|X\circ\phi^Y_s|^2|DY(\phi^Y_s)|,
    \end{gather*}
    where the second inequality follows from $[X,Y]=0$. In particular, the previous inequality and Gronwall's lemma imply that $X(\phi^Y_s(x))=0$ if $X(x)=0$ and one concludes the same result as before. We also remark that the previous inequality holds also for vector fields with bounded divergence, so that the vector field remains zero on the trajectories of the second vector field. 
\end{remark}

%\begin{remark}
  %  In the next subsection we will prove the same results, dropping the continuity assumption on the two vector fields $X$ and $Y$. Clearly, we cannot expect the commutativity of flows in the sense of Equation \eqref{eq:strong:commutativity}, but the weaker one that has been assumed throughout the paper. 
%\end{remark}
In order to prove Proposition \ref{thm:commutativity}, we are now left with proving the following result.
\begin{proposition}\label{prop:commu-nonzeroset}
    In the set $\left(\lbrace X\not=0\rbrace \cap E_X\right)\cap \left(\lbrace Y\not=0\rbrace \cap E_Y\right)$ the two flows commute, in the sense of \eqref{eq:strong:commutativity}.
\end{proposition}
\begin{proof}
    We fix a point $x\in \left(\lbrace X\not=0\rbrace \cap E_X\right)\cap \left(\lbrace Y\not=0\rbrace \cap E_Y\right)$, then, by definition of the two sets, there exist $c_{H(x)} >0$, $c_{K(x)} >0$ such that $|X(y)|\geq c_{H(x)}>0$ for every $y\in \lbrace H=H(x)\rbrace$ and $|Y(y)|\geq c_{K(x)} >0$ for every $y\in \lbrace K=K(x)\rbrace$. %there exist $h\in \tilde E_X$, $k\in \tilde E_Y$ such that $|X(y)|\geq c_h>0$ for every $y\in \lbrace H=h\rbrace$ and $|Y(y)|\geq c_k>0$ for every $y\in \lbrace K=k\rbrace$. 
    In particular, there exists a simple closed curve $C$ with $\phi^X_t(x),\phi^Y_s(x)\in C$ for every $t,s\in \R$ and therefore we can restrict ourselves to study the commutativity of the flows along each curve $C$. We observe that, since the flows of the two vector fields live in the same curve, we can define a function $\tau=\tau(t,y)$ time-periodic satisfying 
    \begin{equation*}
        \phi^X_t(y)=\phi^Y_{\tau(t,y)}(y),\quad\forall y\in C.
    \end{equation*}
    In particular, 
    \begin{equation*}
        \phi^X_t\circ \phi^Y_s(y)=\phi^Y_{\tau(t,\phi^Y_s(y))}\circ \phi^Y_s(y)=\phi^Y_{\tau(t,\phi^Y_s(y))+s}(y)\quad\forall s,t\in\R,\forall y\in C,
    \end{equation*}
    while
    \begin{equation*}
        \phi^Y_s\circ\phi^X_t(y)=\phi^Y_{s+\tau(t,y)}(y)\quad\forall s,t\in\R,\forall y\in C,
    \end{equation*}
    which follow by the semigroup properties of the flows, being the two vector fields autonomous. We claim that $\tau\circ \phi^Y_s=\text{const}$, which gives a proof of the commutativity of the two flows along each curve $C$. Observe that, since $X \cdot Y^\perp=0$, in a $\epsilon$-neighbourhood $C^\epsilon$ of $C$ it holds that
    \begin{equation*}
        X(x)=\alpha_\epsilon(x) Y(x),
    \end{equation*}
    where $\alpha_\epsilon\doteq\frac{X\cdot Y}{|Y|^2}$, defined for every $x\in C_\epsilon$, is continuous, bounded and $\alpha_\epsilon\geq c>0.$ Being $X,Y\in W^{1,p}_{\text{loc}}(C^\epsilon)$, also $\alpha_\epsilon$ inherits the same regularity. In particular, since $\dive (\alpha_\epsilon Y)=\dive X=0$, we get that $\nabla\alpha_\epsilon\cdot Y=0$ a.e. and
    \begin{equation*}
        \frac{d}{ds}\alpha_\epsilon\circ \phi^Y_s(x)= (\nabla\alpha_\epsilon \cdot Y )\circ \phi^Y_s=0.
    \end{equation*}
    In particular, thanks to the continuity assumption, the map $\alpha_\epsilon$ is constant along the trajectories of the vector field $X$ (up to a set of measure zero). Then the function $\tau(t,y)=\alpha_\epsilon(y){t}$ is constant along $\phi^Y_s$ and by Formula \ref{eq:ham:vect:field} we find, for every $y\in C$,
\begin{align*}
    &t=N\int_0^{\ell(C)}\frac{1}{|X(\gamma(s))|}ds+\int_{\gamma^{-1}(y)}^{\bar\tau} \frac{1}{|X(\gamma(s))|}ds\\ &=\frac{1}{\alpha_\epsilon}\left({N}\int_0^{\ell(C)}\frac{1}{|Y(\gamma(s))|}ds+\int_{\gamma^{-1}(y)}^{\bar\tau }\frac{1}{|Y(\gamma(s))|}ds\right)=\tau(t,y),
\end{align*}
where $\gamma(s)$ is the Lipschitz parametrization of the curve $C$ given by Theorem \ref{thm:struct:level:sets} and $\gamma(\bar\tau)=\phi^X_t(y)=\phi^Y_{\tau(t,y)}(y)$.
\end{proof}

\section{Nearly incompressible vector fields}\label{Sect:nearly:incompr}
Here we extend the ideas of the previous section to the more general setting of continuous vector fields $X,Y\in W^{1,p}(\R^2,\R^2)$ with bounded divergence. We recall that, thanks to the boundedness of the divergence, we can assume that the two vector fields $X,Y$ are nearly incompressible (see, for instance, \cite{BianchiniBonicatto}) in the sense of Definition \ref{def:nearly:incomp}. 
 We rewrite the equation defining the Hamiltonian vector field $W$, which follows by Equation \eqref{eq:Lie:bracket} with $[X,Y]=0$:
\begin{equation*}
    W\doteq \dive Y X-\dive X Y \in L^\infty.
\end{equation*}
Clearly, it follows immediately from \eqref{eq:Lie:bracket} that $W$ is a Hamiltonian vector field, with Hamiltonian given by the function
\begin{equation*}
    H_{W}=X\cdot Y^\perp.
\end{equation*}
In particular, being $W\in L^\infty$, the scalar product $X\cdot Y^\perp$ has a Lipschitz continuous representative. 
We start with the following
\begin{proposition}
    Let $X,Y\in W^{1,p}(\R^2,\R^2)\cap L^\infty(\R^2,\R^2)$ with bounded divergence such that $[X,Y]=0$. Then 
    \begin{equation}\label{eq:riccardo}
        X\cdot Y^\perp\circ\phi^X_t= \xi(t) X\cdot Y^\perp,
    \end{equation}
    (and the same for the flow $\phi^Y_s$), where $\xi$ denotes the density of $(\phi^X_t)_\sharp \L^2$ w.r.t. the Lebesgue measure (see Subsection \ref{sect:RLF}). In particular, if $X\cdot Y^\perp\geq 0$, then we have
    \begin{equation}\label{Eq:inequality}
        \frac{1}{C}X\cdot Y^\perp\leq X\cdot Y^\perp\circ\phi^X_t\leq C X\cdot Y^\perp\quad \text{a.e.,}
    \end{equation}
   where $C$ is the constant in \eqref{eq:limitedness}.
\end{proposition}
\begin{proof}
We derive the function $X\cdot Y^\perp\circ\phi^X_t$ and we obtain that
    \begin{equation*}
        \partial_t (X\cdot Y^\perp\circ\phi^X_t)=(\nabla (X\cdot Y^\perp)\cdot X )\circ\phi^X_t=(\nabla^\perp (X\cdot Y^\perp)\cdot X^\perp )\circ\phi^X_t,
    \end{equation*}
    where in the last equality we have simply exploited the definition of $\nabla^\perp$ and $X^\perp$. We now observe that, since $[X,Y]=0$, Equation \eqref{eq:Lie:bracket} yields $\nabla^\perp (X\cdot Y^\perp) =\dive Y X- \dive X Y$ distributionally and therefore
    \begin{equation*}
        (\nabla^\perp (X\cdot Y^\perp)\cdot X^\perp )\circ\phi^X_t=\dive X \circ\phi^X_t (X\cdot Y^\perp) \circ\phi^X_t
    \end{equation*}
    where we also have used the fact that $X \cdot X^\perp =0$. Combining the two previous equalities, we infer
    \begin{equation*}
        \partial_t \left( X\cdot Y^\perp\circ\phi^X_t\right)=\dive X \circ\phi^X_t (X\cdot Y^\perp) \circ\phi^X_t
    \end{equation*}
    On the other hand, by Liouville Theorem (see the discussion in the Preliminaries \ref{S:preliminaries} and in particular Equation \eqref{eq:Liouville}) there holds
    \begin{equation*}
    \label{eq:Liouville2}
    \partial_t \xi=\dive(X)\circ\phi^X_t\xi.
\end{equation*}
    Hence, Equality \eqref{eq:riccardo} follows, since the two functions $X\cdot Y^\perp\circ\phi^X_t$ and $\xi(t) X\cdot Y^\perp$ satisfy the same ODE (with the same initial data). Inequality \eqref{Eq:inequality} is then an immediate consequence of the boundedness of $\xi$ (see in particular inequality \eqref{eq:limitedness}).
\end{proof}
\begin{remark}
    We point out that these computations hold also in the $\BV$ setting.
\end{remark}
For every $\kappa \in \mathbb{R}\setminus \lbrace 0 \rbrace$, we define the sets $\Omega_\kappa=H_W^{-1}\left({C}^{-1}\kappa, C\kappa\right),$ if $\kappa>0$ and 
$\Omega_\kappa=H_W^{-1}\left({C}\kappa, C^{-1}\kappa\right),$ if $\kappa<0$. Clearly, $\Omega_\kappa$ is the union of countably many connected components. In the sequel, when we will refer to $\Omega_\kappa$, for simplicity we will refer to each single connected component. By inequality \eqref{Eq:inequality}, we have that, if $X\cdot Y^\perp(x)=\kappa$, then $\phi^X_t(x)\in \overline{\Omega}_\kappa$.

We are now in a position to prove the main proposition of this section, that is 
\begin{proposition}\label{prop:steady}
    Let $X,Y\in L^\infty(\R^2,\R^2)\cap W^{1,p}(\R^2,\R^2)$ with bounded divergence such that $[X,Y]=0$. Then for every $\kappa\in \R\setminus \lbrace{0}\rbrace$, the vector fields $X\llcorner_{\Omega_\kappa},Y\llcorner_{\Omega_\kappa},$ are steady nearly incompressible. Moreover, they have the same Lipschitz and bounded density $\rho_\kappa:\Omega_\kappa\rightarrow\R$ given by the function
    \begin{equation*}
        \rho_\kappa=\frac{1}{H_W},
    \end{equation*}
    that is 
    \begin{equation*}
        \dive(\rho_\kappa (X-Y))=0, \quad\text{in }\mathcal D'(\Omega_\kappa).
    \end{equation*}
\end{proposition}
\begin{proof}
Since we are in the set $\Omega_\kappa$, the function $\rho_\kappa$ of the statement is well-defined, bounded (from above and below) and Lipschitz, being $X\cdot Y^\perp$ Lipschitz. In particular, by Equation \eqref{eq:Lie:bracket} with $[X,Y]=0$, it holds
\begin{equation*}
    \nabla^\perp(X\cdot Y^\perp)\cdot Y^\perp=\dive Y (X\cdot Y^\perp), 
\end{equation*}
and therefore
\begin{align*}
    \dive (\rho_\kappa Y)&=\dive\left(\frac{1}{H_W} Y\right)\\&=
    \frac{1}{H_W}\dive Y-\frac{1}{H_W^2}\nabla\left(X \cdot Y^\perp\right)\cdot Y\\&=
    \frac{1}{H_W}\dive Y-\frac{1}{H_W^2}\nabla^\perp\left(X \cdot Y^\perp\right)\cdot Y^\perp\\&=
    \frac{1}{H_W}\dive Y-\frac{1}{H_W^2}\dive Y H_W=0.
\end{align*}
The same computation holds for the vector field $X$, being
\begin{equation*}
    \nabla^\perp(X\cdot Y^\perp)\cdot X^\perp=-\dive X (Y\cdot X^\perp)=\dive X (X\cdot Y^\perp).
\end{equation*}
\end{proof}
We define now the set
\begin{equation*}
    \Omega=\cup_{\kappa\in \Q\setminus \lbrace 0\rbrace} \Omega_\kappa,
\end{equation*}
which is an open set such that $\Omega^c=\lbrace H_W=0\rbrace$. Observe that, thanks again to inequality \eqref{Eq:inequality}, if $x\in \Omega$, then $\phi^X_t(x)\in \Omega$ (and the same for the flow of the vector field $Y$), so that the set $\Omega$ is an invariant set for the two flows.

Taking advantage of the previous results, we can easily adapt the ideas of \cite{Eliodiff} to prove the following
\begin{proposition}
    Let $X,Y$ be vector fields satisfying the hypotheses of the two previous propositions. Moreover, we assume that they are continuous. Then, for every $t,s$ the RLFs of $X,Y$ have representative $\phi^X_t,\phi^Y_s\in W^{1,p}(\Omega)\cap \mathcal C^0(\Omega).$ 
\end{proposition}
\begin{proof}
    We fix $\kappa\in\R\setminus\lbrace 0\rbrace$ and we consider $\rho_\kappa$ as in the statement of Proposition \ref{prop:steady}. Then, as $\rho_\kappa$ is Lipschitz and bounded, the vector field $\rho_\kappa X$ belongs to the space $W^{1,p}_{\text{loc}}(\R^2,\R^2)$ and is continuous on the set $\Omega_\kappa$. By Diperna-Lions Theory there exists a unique RLF 
    \begin{equation*}
        \phi^\kappa_t\doteq \phi_t^{\rho_\kappa X}
    \end{equation*}
    solving the Cauchy Problem
    \begin{equation*}
        \begin{cases}
            \partial_t \phi^\kappa_t(x)=(\rho_\kappa X)(\phi^\kappa_t(x)),\\
            \phi^\kappa_{t=0}(x)=x.
        \end{cases}
    \end{equation*}
    Note that the trajectories of the vector field $\rho_\kappa X$ live in the set $\Omega$, being this vector field Hamiltonian on $\Omega_\kappa$. By \cite{Eliodiff}, $\phi^\kappa_t\in W^{1,p}_{\text{loc}}(\Omega_\kappa)$. Indeed, we first consider a countable cover of $\Omega_\kappa$ by open balls, namely $\Omega_\kappa\subset\cup_j B_{r_j}(x_j)$, where $r_j>0$ and $x_j\in \Omega_\kappa$. The countable cover can be chosen in such a way that $X\cdot Y^\perp \in (C^{-1}\kappa-\delta,C\kappa+\delta),$ for some $\delta>0$ sufficiently small, thanks to the continuity of the two vector fields. Finally, we denote by $\Omega^\delta_\kappa$ the set $H_W^{-1}((C^{-1}\kappa-\delta,C\kappa+\delta))$. To prove the assertion we just need to show that $\phi^\kappa_t\in W^{1,p}(B_{r_j}(x_j))$. 
    This clearly holds true, since in each ball $B_{r_j}(x_j)$ there exists a Hamiltonian $H_{j}$ such that $\rho_\kappa X=\nabla^\perp H_j$. We then denote $\tilde \Omega_n= H_\kappa^{-1}(\lbrace h: \min_{H_j^{-1}(h)} |\rho_\kappa X|>\frac{1}{n}\rbrace).$  In particular, since $Y$ is bounded, we have
    \begin{equation*}
        \bar {B}_{r_j}(x_j)\subset\tilde\Omega_n.
    \end{equation*}
    The sets $\tilde \Omega_n$ are those defined in \cite{Eliodiff} (Section 3), where the flow has been proved to be Sobolev and therefore we have proved that $\phi^\kappa_t\in W^{1,p}(B_{r_j}(x_j))$. 
    
    We now claim that, if we define $\tau(t,x)$ as the solution of the following ODE system
    \begin{equation}\label{eq:sistema:ODE}
    \begin{cases}
        \frac{d\tau(t,x)}{dt}=\frac{1}{\rho_\kappa(\phi^k_{\tau(t,x)}(x))}=X\cdot Y^\perp(\phi^k_{\tau(t,x)}(x)), \\
        \tau(0,x)=0,
        \end{cases}
    \end{equation}
   for every $x \in \Omega_\kappa$, then there holds
    \begin{equation}\label{eq:par-phi}
        \phi^X_t=\phi^\kappa_{\tau(t,x)},
        \end{equation}
        (where $\phi^X_t$ is the unique RLF of $X$).
    Indeed, deriving the right hand side of \eqref{eq:par-phi} and employing \eqref{eq:sistema:ODE}, we obtain
    \begin{equation*}
        \frac{d}{dt} \phi^\kappa_{\tau(t,x)}(x)=\rho_\kappa (\phi^\kappa_{\tau(t,x)}(x)) X(\phi^\kappa_{\tau(t,x)}(x))\frac{1}{\rho_\kappa(\phi^\kappa_{\tau(t,x)}(x))},
    \end{equation*}
     and therefore \eqref{eq:par-phi} immediately follows by uniqueness of the RLF. Notice that $X\cdot Y^\perp\in \text{Lip}(\R^2,\R)\cap L^\infty(\R^2,\R).$ Moreover, $\phi^\kappa_t\in W^{1,p}_{\text{loc}}(\Omega_\kappa)$ for all $t$.

     This in particular implies that $\phi^X_t\in W^{1,p}_{\text{loc}}(\Omega_\kappa),$ and thus $\phi^X_t\in W^{1,p}_{\text{loc}}(\Omega),$ once the function $\tau$ is proved to exist and have Sobolev regularity. To prove this last statement, we observe that the function $t\rightarrow X\cdot Y^\perp \circ \phi^\kappa (t,\cdot)$ is Lipschitz uniformly w.r.t. $x\in\Omega_\kappa$ and it is bounded, giving global existence and uniqueness. Moreover, $x\rightarrow X\cdot Y^\perp \circ \phi^\kappa (t,x)$ is in $W^{1,p}_{\text{loc}}(\Omega_\kappa)$ for every time $t$ . This also gives the Sobolev regularity of the function $\tau$. 
   
    Since $\phi^X_t\in W^{1,p}_{\text{loc}}(\Omega_\kappa)$ for every $\kappa$, we conclude that $\phi^X_t\in W^{1,p}_{\text{loc}}(\Omega).$ 
\end{proof}

We are finally ready to prove our main theorem:
\begin{theorem}\label{thm:main:n:i}
    Let $X,Y\in W^{1,p}(\R^2,\R^2)\cap C^0(\R^2,\R^2)$ be two bounded vector fields with $\dive X,\dive Y \in L^\infty(\R^2,\R^2).$ Then it holds
\begin{equation*} \qquad [X,Y]=0 \iff \phi_t^X \circ \phi_s^Y=\phi_s^Y \circ \phi_t^X , \quad \forall t,s \in \R,\quad\text{for a.e. in }\R^2.
    \end{equation*}
\end{theorem}
\begin{proof}[Proof of Theorem \ref{thm:main:n:i}]

    Since $\phi^X_t,\phi^Y_s\in W^{1,p}(\Omega)$, by the result in \cite{ColomboTione} we have the commutativity of the two flows in that set, being invariant for the flows (Proposition \ref{thm:eq:riccardo}). In particular, if $\L^2(\Omega^c)=0$, the proof is concluded. 
    
    If not,  we observe that $W=0$ a.e. in the set $\Omega^c$, (since the gradient of Sobolev functions is strongly local), which in particular implies that
    \begin{equation}\label{eq:uguaglianza:div}
        \dive X Y=\dive Y X\quad\text{a.e. in }\Omega^c,
    \end{equation} 
    We prove now the commutativity of flows in the open set $\lbrace X=0\rbrace^c$ (while in the set $\lbrace X=0\rbrace$ it is a consequence of Gronwall's lemma as done in Remark \ref{rmk:gronwall}). The condition $X\cdot Y^\perp=0$ implies that in the set $\lbrace X=0\rbrace^c$ there exists a continuous function $\alpha$  with $Y=\alpha X$ and $\alpha\in W^{1,p}_{\text{loc}}(\lbrace X=0\rbrace^c)$ defined as:
    \begin{equation*}
        \alpha(x)=\frac{X\cdot Y}{|X|^2}.
    \end{equation*}
    Moreover, equation \eqref{eq:uguaglianza:div}, implies that $(\nabla\alpha\cdot X) X=0$, which in turn yields $\nabla \alpha\cdot X=0$ a.e. in the set $\lbrace Y=0\rbrace^c\cup \lbrace X=0\rbrace^c$.
    
    Always by a Gronwall-type argument, condition $ \nabla \alpha \cdot X=0$, together with the continuity of the function $\alpha$, implies that the the function $\alpha$ is constant along the trajectories of the vector field $X$ (up to a set of measure zero). Finally we observe that, if $t$ is fixed, then 
   \begin{align*}
       \phi^X_t(x)&=x+\int_0^t X(\phi^X_\tau(x))d\tau=x+\frac{1}{\alpha}\int_0^t \alpha X(\phi^X_\tau(x))d\tau=x+\frac{1}{\alpha}\int_0^t Y(\phi^X_\tau(x))d\tau\\&= x+\int_0^{\frac{t}{\alpha}}Y(\phi^X_{\alpha\tau}(x))d\tau,
   \end{align*}
   which yields $\phi^X_t=\phi^Y_{\frac{t}{\alpha}}$. The commutativity follows by the same argument of the proof of Proposition \ref{prop:commu-nonzeroset}.
\end{proof}

\section{Some remarks on BV vector fields}\label{Sect:BV}
 %In this section we discuss possible extension to the $\BV$ case.
This section is devoted to the proof of Proposition \ref{thm:comm-BV}, which we carry out by adapting the one in \cite{ColomboTione} to our case. We first briefly discuss the definition of Lie bracket we rely on in this work and we show that this definition is still suitable if we consider $\BV$ vector fields. More precisely, we point out that, if we identify $X,Y$  with their precise representative, the Lie bracket defined in \ref{eq:def:lie:brack} reduces to the distributional one given in \cite{ColomboTione}, namely
    \begin{equation*}
        [X,Y]\doteq DYX-DXY,\quad\text{distributionally}.
    \end{equation*}
    Indeed, by the Leibniz rule \eqref{eq:leibniz-rule}, we can rewrite the first term in \eqref{eq:def:lie:brack} as follows
    \begin{equation*}
  (X,D \varphi Y) =\sum_{i,j=1}^n \partial_j \left(     X^i \varphi^i Y^j\right)-(X^*,\varphi) \dive Y-(DX Y^*, \varphi),
    \end{equation*}
    where $\varphi$ is a smooth test function. 
 Integrating the previous expression and keeping in mind that $X=X^*$ $\L^n$-a.e., we obtain
\begin{equation*}
    \int_{\R^n}(X,D \varphi Y) dz = - \int_{\R^n}(X,\varphi) \dive Y dz - \sum_{i,j=1}^n\int_{\R^n} \varphi^i (Y^*)^j dD_jX^i .
\end{equation*}
Similarly, we have
\begin{equation*}
    \int_{\R^n}(Y,D \varphi X) dz = - \int_{\R^n}(Y,\varphi) \dive X dz - \sum_{i,j=1}^n\int_{\R^n} \varphi^i (X^*)^j d D_jY^i,
\end{equation*}
and therefore
\begin{equation*}
    \int_{\R^n}([X,Y], \varphi) dz = \sum_{i,j=1}^n\int_{\R^n} \varphi^i (X^*)^j d D_jY^i-\sum_{i,j=1}^n\int_{\R^n} \varphi^i (Y^*)^j dD_jX^i.
\end{equation*}
For this reason, in the following, we will always identify $X$ and $Y$ with their precise representatives.

\smallskip

We finally remark that, in the smooth setting, the commutativity of the flows $\phi_t^X$, $\phi_s^Y$ is equivalent to the following property
\begin{equation}\label{eq:der-smooth}
    D \phi_t^X(z)Y(z)=Y(\phi_t^X(z)), \quad \forall  z \in \R^n, t \in \R,
\end{equation}
{where the equality is obtained deriving in time both left and right hand side and observing that, thanks to the assumption $[X,Y]=0$, they satisfy the same differential equation with the same initial data. In the $\BV$ setting, it is not at all clear the meaning that should be given to this equivalence.}
As in \cite{ColomboTione}, Equation \eqref{eq:der-smooth} motivates the definition of the following distribution, already introduced in \cite{ColomboTione}:
\begin{equation}\label{eq:distr1}
    T_t[X,Y](\varphi) \doteq \int_{\R^n}(Y (\phi_t^X),\varphi)+(\phi_t^X, D\varphi Y)+ \dive Y (\varphi,\phi_t^X)dz, \quad \varphi \in C_c^\infty(\R^n,\R^n)
\end{equation}
which clearly equals $Y(\phi_t^X(z))-D \phi_t^X(z)Y(z)$ if $X$ and $Y$ are smooth vector fields. We consider also the following family of distributions as in \cite{ColomboTione}
\begin{equation}\label{eq:distr2}
    T_{t,s}[X,Y](\varphi) \doteq \int_{\R^n}(Y (\phi_t^X \circ \phi_s^Y),\varphi)+(\phi_t^X \circ \phi_s^Y, D\varphi Y)+ \dive Y (\varphi,\phi_t^X \circ \phi_s^Y)dz, \quad \varphi \in C_c^\infty(\R^n,\R^n)
\end{equation}
which represents the function $Y(\phi_t^X \circ \phi_s^Y)(z)-D (\phi_t^X \circ \phi_s^Y)(z)Y(z)$ in the smooth setting. The distributions defined in \eqref{eq:distr1} and \eqref{eq:distr2} play a crucial role in our analysis since they serve as an intermediate step in proving the following proposition.

\begin{proposition}\label{prop:commutativity}
 Then the following hold:
\begin{equation}
    \phi^X_t\circ\phi^Y_s=\phi^Y_s\circ\phi^X_t,\quad\forall t,s\in\R\text{ and for a.e. }x\in\R^n\implies T_{t,s}=0 \quad\forall t,s\in\R\implies T_{t}=0 \quad\forall t\in\R.
\end{equation}
In addition, $T_t=0$ for every $t \in \R$ implies $[X,Y]=0$.
\end{proposition}
\begin{proof}
This part of the proof is identical to the one in \cite{ColomboTione} and we rewrite it for completeness. The only changes are in the proof of the last implication, where we make use of the definition of distributional Lie Bracket.
 We first point out that, if $T_{t,s}=0$ for every $t,s \in \R$, then it is sufficient to evaluate $T_{t,s}$ at $s=0$ to obtain $T_t=0$ for every $t \in \R$.

\noindent
\textbf{Proof of} $\phi_t^X \circ \phi_s^Y =\phi_s^Y \circ \phi_t^X, \forall t,s \, \implies \, T_{t,s}=0, \forall t,s$. As a first step, we aim at proving that
\begin{equation}\label{eq:double-implication-flows}
    \phi_t^X \circ \phi_s^Y =\phi_s^Y \circ \phi_t^X, \forall t,s \, \implies \, \partial_s (\phi_t^X \circ \phi_s^Y )=Y(\phi_t^X \circ \phi_s^Y ), \, \text{for a.e. $s \in \R$}.
\end{equation}
Indeed, if $ \phi_t^X \circ \phi_s^Y =\phi_s^Y \circ \phi_t^X, \forall t,s$, then $s \mapsto  \phi_t^X \circ \phi_s^Y$ is an absolutely continuous function which satisfies
\begin{equation*}
     \partial_s (\phi_t^X \circ \phi_s^Y )= \partial_s ( \phi_s^Y \circ \phi_t^X )=Y( \phi_s^Y \circ \phi_t^X)=Y(\phi_t^X \circ \phi_s^Y ), \quad \text{for a.e. $s \in \R$}.
\end{equation*}
We use this equivalence \eqref{eq:double-implication-flows} to prove the desired claim. We first observe that equation $ \partial_s (\phi_t^X \circ \phi_s^Y )=Y(\phi_t^X \circ \phi_s^Y )$ needs to be understood in the following sense
\begin{equation*} 
    -\int_{\R^{n+1}}\left(\partial_s\varphi(s,z),\phi_t^X \circ \phi_s^Y\right) \, ds \, dz=\int_{\R^{n+1}}\left(Y(\phi_t^X \circ \phi_s^Y), \varphi(s,z)\right) \, ds \, dz, \quad \forall \varphi \in C_c^\infty(\R^{n+1}).
\end{equation*}
We now perform the change of variable $\tilde{z}=\phi_s^Y(z)$ and rewrite the previous equation as follows
\begin{align}\label{eq:change-var-tildez}
-\int_{\R^{n+1}}\left(\partial_s\varphi(s,\phi_{-s}^Y(\tilde{z})),\phi_t^X (\tilde{z})\right)\, \xi(-s,\tilde z) \, ds \, d\tilde{z}=\int_{\R^{n+1}}\left(Y(\phi_t^X \circ \phi_s^Y), \varphi(s,z)\right) \, ds \, dz, 
    %\quad \forall \varphi \in C_c^\infty(\R^{n+1}).
\end{align}
We now exploit the change of variable formula \eqref{eq:change-of-var} to compute
\begin{equation*}
\int_{\R^{n+1}}\left(\partial_s\left( \varphi(s,\phi_{-s}^Y(\tilde{z}))\right),\phi_t^X (\tilde{z})\right)\, \xi(-s,\tilde z) \, ds \, d\tilde{z}=\int_{\R^{n+1}}\left(\varphi(s,\phi_{-s}^Y(\tilde{z})),\phi_t^X (\tilde{z})\right)\dive Y \circ \phi_{-s}^Y\, \xi(-s,\tilde z) \, ds \,.
\end{equation*}
Taking advantage of the previous equation, together with
\begin{equation*}
\partial_s\left(\varphi(s,\phi_{-s}^Y(\tilde{z}))\right)=\partial_s\varphi(s,\phi_{-s}^Y(\tilde{z}))- D\varphi(s,\phi_{-s}^Y(\tilde{z}))Y(\phi_{-s}^Y(\tilde z)),
\end{equation*}
we can rewrite \eqref{eq:change-var-tildez} as follows
\begin{align*}
&-\int_{\R^{n+1}}\left(D\varphi(s,\phi_{-s}^Y(\tilde{z})) Y\left(\phi_{-s}^Y(\tilde{z})\right),\phi_t^X (\tilde{z})\right)\, \xi(-s,\tilde z) \, ds \, d\tilde{z}\\&\qquad-\int_{\R^{n+1}}\left(\varphi(s,\phi_{-s}^Y(\tilde{z})),\phi_t^X (\tilde{z})\right)\dive Y \circ \phi_{-s}^Y\, \xi(-s,\tilde z) \, ds \, d\tilde{z}=\int_{\R^{n+1}}\left(Y(\phi_t^X \circ \phi_s^Y), \varphi(s,z)\right) \, ds \, dz.
    %\quad \forall \varphi \in C_c^\infty(\R^{n+1}).
\end{align*}
Performing back the change of variable, we obtain
\begin{align*}
&-\int_{\R^{n+1}}\left(D\varphi(s,z) Y,\phi_t^X \circ \phi_s^Y \right) \, ds \, dz\\&\qquad-\int_{\R^{n+1}}\left(\varphi(s,z),\phi_t^X \circ \phi_s^Y \right)\dive Y \, ds \,dz=\int_{\R^{n+1}}\left(Y(\phi_t^X \circ \phi_s^Y), \varphi(s,z)\right) \, ds \, dz.
    %\quad \forall \varphi \in C_c^\infty(\R^{n+1}).
\end{align*}
If we now choose test functions of the form $\varphi=\alpha(s)\beta(z)$, the previous equality is equivalent to state that for every $\beta \in C_c^\infty(\R)$, $t \in \R$ and for a.e. $s \in \R$, there holds
\begin{align}\label{eq:beta}
-\int_{\R^{n}}\left(D\beta(z) Y,\phi_t^X \circ \phi_s^Y \right) \, dz-\int_{\R^{n}}\left(\beta(z),\phi_t^X \circ \phi_s^Y \right)\dive Y  \,dz=\int_{\R^{n}}\left(Y(\phi_t^X \circ \phi_s^Y), \beta(z)\right) \, dz.
    %\quad \forall \varphi \in C_c^\infty(\R^{n+1}).
\end{align}
We finally prove that \eqref{eq:beta} holds true for every $s \in \R$. To this end, we first observe that the terms that do not involve $\dive Y$ are absolutely continuous in $s$, as can be seen performing again the change of variable $\tilde{z}=\phi_s^Y(z)$ and using the dominated convergence theorem. At the same time, the term involving $\dive Y$ in equality \eqref{eq:beta} is also continuous in $s$, as, for every $\epsilon >0$, we can rewrite it as
\begin{align*}
  &-\int_{\R^{n}}\left(\beta(z),\phi_t^X \circ \phi_s^Y \right)\dive Y  \,d z =\\ &\qquad
  -\int_{\R^{n}}\left(\beta(z),\left(\phi_t^X \ast \rho_\epsilon \right)\circ \phi_s^Y \right)\dive Y  \,d z -\int_{\R^{n}}\left(\beta(z),\left(\phi_t^X-\phi_t^X \ast \rho_\epsilon \right)\circ \phi_s^Y \right)\dive Y  \,d z,
\end{align*}
where $\rho_\epsilon$ denotes, as above, a mollifier. We now point out that the first term in the previous equality is continuous in $s$, while the second term converges to zero (as $\epsilon \to 0$) locally uniformly in $s$. Thus, equality \eqref{eq:beta} holds true for every $s \in \R$. We eventually conclude observing that \eqref{eq:beta} is equivalent to $T_{t,s} \equiv 0$ for every $t,s \in \R$.

\medskip
\noindent
\textbf{Proof of} $ T_t =0, \forall t \, \Rightarrow \, [X,Y]=0.$ We aim at showing that
\begin{equation}\label{eq:deriv-T_t}
    \frac{d}{dt}T_t[X,Y]\Big|_{t=0}=[X,Y],
\end{equation}
which clearly implies the desired implication. To this end, we first observe that, performing the change of variable $\tilde{z}=\phi_t^X(z)$ and exploiting again \eqref{eq:change-of-var}, we can write
\begin{align*}
\frac{d}{dt}\int_{\R^{n+1}}(Y (\phi_t^X(z)),\varphi(s,z))\,dz\,ds=&-\int_{\R^{n+1}}(Y(\tilde{z}),D \varphi(s,\phi_{-t}^X(\Tilde{z}))X(\phi_{-t}^X(\Tilde{z})))\,\xi(-t,\Tilde{z})\,d\Tilde{z}\,ds\\
&\quad - \int_{\R^{n+1}}(Y(\tilde{z}), \varphi(s,\phi_{-t}^X(\Tilde{z})))\dive X(\phi_{-t}^X(\Tilde{z}))\,\xi(-t,\Tilde{z})\,d\Tilde{z}\,ds.
\end{align*}
for every $\varphi \in C_c^\infty(\R^{n+1},\R^n)$. Performing back the change of variable in the previous integrals, we infer that, for every $\varphi \in C_c^\infty(\R^{n+1},\R^{n+1})$, there holds
\begin{align}\label{eq:deriv-Y-phi}
\frac{d}{dt}\int_{\R^{n+1}}(Y (\phi_t^X),\varphi)dzds=&-\int_{\R^{n+1}}(Y(\phi_t^X(z)),D \varphi(z)X(z))dzds- \int_{\R^{n+1}}(Y(\phi_t^X(z)), \varphi(z))\dive X dzds.
\end{align}
In particular, with the choice of functions $\varphi(s,z)=\alpha(s)\beta(z)$ as done in \cite{ColomboTione}, for every $\beta\in\mathcal{C}^\infty_c(\R^{n},\R^n)$, it holds that 
\begin{align}\label{eq:deriv-Y-phi}
\frac{d}{dt}\int_{\R^{n}}(Y (\phi_t^X),\beta)dz=&-\int_{\R^{n}}(Y(\phi_t^X(z)),D \beta(z)X(z))dz- \int_{\R^n}(Y(\phi_t^X(z)), \beta(z))\dive X dz.
\end{align}

We next observe that, for any test function $\varphi \in C^\infty_c(\R^n,\R^n)$, we can differentiate \eqref{eq:distr1} under the integral sign with respect to $t$ and exploit \eqref{eq:deriv-Y-phi} to obtain
\begin{equation}\label{eq:der-distr}
\begin{split}
    \frac{d}{dt} T_t[X,Y](\varphi) = -\int_{\R^n}(Y(\phi_t^X),D \varphi X)dz- \int_{\R^n}(Y(\phi_t^X), \varphi)\dive X dz\\+\int_{\R^n}(X (\phi_t^X), D\varphi Y)dz+ \int_{\R^n}\dive Y (\varphi,X ( \phi_t^X))dz.
    \end{split}
\end{equation}
Since $T_t[X,Y]=0$ for every $t$, Claim \eqref{eq:deriv-T_t} then follows by evaluating \eqref{eq:der-distr} at $t=0$ and by keeping in mind our definition of Lie bracket in \eqref{eq:def:lie:brack}. This concludes the proof of Proposition \ref{prop:commutativity}.
\end{proof}
We conclude by pointing out that the expression \eqref{eq:der-distr} should encode, in our opinion, the distributional version of the function $[X,Y]\circ\phi^X_t$, that we leave as an open question.

\printbibliography

@article{ABC,
	author = "{Alberti G., Bianchini S., Crippa G.}",
	title = "{A uniqueness result for the continuity equation in two
dimensions}",
	journal= "Journal of the European Mathematical Society",
	number= " 16(2) ",
	year = "2014",
	pages= "201–234",
	}

@book{AFP,
    title = "{Functions of Bounded Variations and Free Discontinuity Problems}",
    author = "{Ambrosio L., Fusco N., Pallara D. }",
    isbn = "{0198502451}",
    series = {International series of monographs on physics},
    year = "2000",
    publisher = "{Oxford University Press}"
}

@article{Ambrosio:Luminy,
	author = "{Ambrosio L.}",
	title = "{Lecture notes on transport equation and Cauchy problem for BV vector fields and applications}.",
	year = "2004",
}

@article{Ambrosio:BV,
	author = "{Ambrosio L.}",
	title = "{Transport equation and Cauchy problem for BV vector fields.}",
	journal = "Inventiones Mathematicae",
	number = "158",
	pages = "227-260",
	year = "2004",
}

@book{Arnold:khesin,
	author = "{Arnold V.I., Khesin B. A.}",
	title = "{Topological methods in hydrodynamics}",
	journal = "Springer",
	year = "{1998,2021}",
}

@article{BianchiniBonicatto,
	author = "{Bianchini. S., Bonicatto P.}",
	title = "{ A uniqueness result for the decomposition of vector fields in $R^d$}.",
	journal = "Invent. Math.",
 pages = "255–393",
    volume="220(1)",
	year = "2020",
}

@article{BianchiniBonicattoGusev,
	author = "{Bianchini. S., Bonicatto P., Gusev N. A.}",
	title = "{Renormalization for Autonomous Nearly Incompressible BV Vector Fields in Two Dimensions}.",
	journal = "SIAM Journal on Mathematical Analysis",
    volume="48(1)",
	year = "2016",
}

@article{BianchiniGusev,
	author = "{Bianchini. S., Gusev N. A.}",
	title = "{Steady nearly incompressible vector fields in $2D$: chain rule and renormalization}.",
	journal = "Archive for Rational Mechanics and Analysis",
 pages = "451-505",
    volume="222(2)",
	year = "2016",
}

@article{BourgainKK,
	author = "{ Bourgain J., Korobkov M. ,  Kristensen  J.}",
	title = "{On the Morse-Sard property and level sets of Sobolev and BV
functions.}",
	journal = "{Rev. Mat. Iberoam.}",
    pages = "1–23",
    volume="29(1)",
	year = "2013",
}

@article{BrenierMHD,
	author = "{Brenier Y.}",
	title = "{Topology-preserving diffusion of divergence-free vector fields and magnetic relaxation.}",
	journal= "{Communications in Mathematical Physics}",
	volume= " 330 ",
	year = "2014",
	pages= "757–770",
	}

@article{ColomboTione,
	author = "{Colombo M., Tione R.}",
	title = "{On the commutativity of flows of rough vector fields}",
	journal = "{Journal des Mathematiques Pures et appliqu\`es}",
	pages = "294-312 ",
	year = "2021",
    volume="159",
}

@article{Marconi_Bonicatto_poly,
	author = "{Bonicatto P., Marconi E.}.",
	title = "{Regularity estimates for the flow of BV autonomous divergence free vector fields in $\R^2$}",
	journal="{Communications in Partial Differential Equations}",
    pages="2235-2267 ",
    volume="{46:12}",
    year="2021",
}

@article{DeLellisNotes,
	author = "{De Lellis C}.",
	title = "{Notes on hyperbolic systems of conservation laws and transport equations}.",
	journal= "Handb. Differ. Equ.",
	number= "3",
	year = "2007",
	pages= "277-382",
	}

@article{DelNinBonicatto,
author="{Bonicatto P., Del Nin G., Rindler F.}",
title="{ransport of currents and geometric Rademacher-type theorems}",
year="2022",
}

@article{DiPerna:Lions,
	author = "{Diperna R. J., Lions P. L.}",
	title = "{Ordinary differential equations, transport theory and Sobolev spaces}.",
	journal = "Inventiones Mathematicae",
	number = "98",
	pages = "511-547",
	year = "1989",
}

@article{jab16,
	author = "{Jabin P.E.}",
	title = "{Critical non-Sobolev regularity for continuity equations with rough velocity fields}",
	journal = "{Journal of Differential Equations}",
	volume = "{260(5)}",
	pages = "{4739-4757}",
	year = "2016",
}

@article{Eliodiff, 
	author = "{Marconi E.}",
	title = "{Differentiability properties of the flow of 2d autonomous vector fields}",
	journal = "{Journal of Differential Equations}",
	volume = "301",
	pages="330-352",
    year = "2021",
	pages = "874-920",
}

@article{RigoniStepanov,
author="{Rigoni C., Stepanov E., Trevisan D.}",
title="{Lie brackets of nonsmooth vector fields and commutation of their flows}",
journal="{Journal of the London Mathematical Society}",
year="2022",
}

@article{RampazzoSussmann,
	author = "{Rampazzo F, Sussmann H.J.}",
	title = "{Commutators of flow maps of nonsmooth vector fields}",
	journal = "Journal of Differential Equations",
	number = "1",
	pages = "134-175",
	year = "2007",
}

@article{Volpert,
    author= "{Vol'pert A. I.}",
    title= "{Spaces BV and quasi-linear equations}",
    journal = "{Math. USSR Sb.}",
volume = "17",
    pages = "225-267",
    year = "1967",
 }
%\bibliography{flows}

\end{document}